\documentclass{article}[12pt,draft]

\usepackage{amsmath}
\usepackage{amsthm}
\usepackage{amsfonts}

\usepackage{fullpage}
\usepackage{setspace}
\usepackage{epsfig}
\usepackage{graphicx}
\usepackage{epstopdf}
\newtheorem{corr}{Corollary}
\newtheorem{lem}{Lemma}

\newtheorem{df}{Definition}
\newtheorem{remarks}{Remark}
\newtheorem{example}{Example}
\newtheorem{thm}{Theorem}

\usepackage{color}

\newcommand{\C}{\nu_H}

\newcommand{\RM}{\mathbb{R}}

\newcommand{\AM}{{\bf A}}
\newcommand{\MM}{\,\mbox{\bf M}}

\newcommand{\cof}{\operatorname{cof}}

\newcommand{\rank}{\operatorname{rank}}

\title{A geometric method for eigenvalue problems with low rank perturbations}

\author{Thomas J. Anastasio\thanks{Department of Molecular and Integrative Physiology and Beckman Institute,
University of Illinois Urbana-Champaign, Urbana, 
IL 61820.}
\and Andrea K. Barreiro\thanks{ Department of Mathematics,
Southern Methodist University, Box 750156, Dallas, TX 75275} 
\and Jared C. Bronski\thanks{Department of Mathematics,
University of Illinois Urbana-Champaign, 1409 W. Green St., Urbana, 
IL 61801.}}
\begin{document}
\maketitle

\begin{abstract}We consider the problem of finding the spectrum 
of an operator taking the form of a low-rank (rank one or two) non-normal 
perturbation of a well-understood operator, motivated by a number of problems of applied interest which take this form. 
We use the fact that the system is a low rank
perturbation of a solved problem, together with a 
simple idea of classical differential geometry (the envelope of a family of 
curves) to completely analyze the spectrum.   
We use these techniques to analyze three problems of this form: a 
model of the oculomotor integrator due to Anastasio and Gad\cite{ag07}, a continuum integrator model, and a nonlocal model of phase separation due to Rubinstein and Sternberg\cite{RS92}. 

\end{abstract}

\noindent
\textbf{Keywords:} Bifurcation theory, Aronszajn-Krein formula, Rank one perturbations



\section*{Introduction}

In this paper we analyze eigenvalue problems of the following form
\begin{equation}
\widetilde{\bf M}\vec w = {\bf M} \vec w + \rho_1 \vec f_1 \langle \vec g_1, \vec w\rangle + \rho_2 \vec f_2 \langle \vec g_2, \vec w\rangle = \lambda \vec w 
\label{eqn:Rank1}
\end{equation}
with $\rho_1$ and $\rho_2$ parameters, $\vec f_i, \vec g_i$ fixed vectors 
 and ${\bf M}$ is an operator with known spectrum. While Eqn. \eqref{eqn:Rank1} might appear very specific, we are aware of a number of interesting eigenvalue 
problems which take this form. These include:
\begin{itemize}
\item A model due to Anastasio and Gad of the behavior of the oculomotor 
integrator\cite{ag07}.
\item A non-local Allen-Cahn model due to Rubinstein and Sternberg for phase separation\cite{RS92}.
\item The stability problem for spike solutions to activator-inhibitor models 
in the limit of slow activator diffusion\cite{F94,IW02}.
\item Stability for models of runaway ohmic heating\cite{C81,L95I,L95II} and  microwave heating \cite{BK98}.
\item Stability for stationary solutions of model for phytoplankton growth.\cite{DH10}
\end{itemize}
Four of the models take the form of a reaction-diffusion equation with a 
nonlocal term. The study of the stability of stationary solutions of such 
models naturally leads to an eigenvalue problem which takes the form of a 
self-adjoint Sturm-Liouville operator plus a finite rank perturbation
coming from the nonlocal term.   
Freitas\cite{F94} has considered a similar problem and has 
some related results for a single perturbation (rather than a two parameter 
family); Bose and Kriegsmann\cite{BK98} have some other related results, 
mainly in the case where $\vec f_i = \vec g_i$ where the problem is
self-adjoint.
We refer the interested reader to the review paper of
Freitas\cite{F99}, which details a number of models whose stability
problems take the form of a ``nice'' operator with a low rank
(typically rank one) perturbation.

In all of these problems the eigenvalue problem arises in the study of the 
stability of a particular steady state. In this situation one is typically interested in 
understanding qualitative properties of the spectrum as a function of the parameters $(\rho_1, \rho_2)$. 
In particular one might wish to understand, for any particular pair $(\rho_1, \rho_2)$,
\begin{itemize}
 \item how many eigenvalues are in the right half-plane, and
\item  how many eigenvalues are real (vs. complex). 
\end{itemize}
Here we give a direct way 
to construct a phase-diagram in the $(\rho_1,\rho_2)$ plane which answers
these questions. 

The main approach that we take here is to exploit the low rank nature of the perturbations,  
along with some geometric constructions for the quantities of interest. Since the eigenvalues will vary continuously as a function of the parameters $(\rho_1, \rho_2)$, the quantities noted above are constant on open sets, with the boundary of these sets being respectively 
\begin{itemize}
 \item the set of $(\rho_1, \rho_2)$ for which $\widetilde {\bf M}$ has a purely complex eigenvalue $\lambda = \imath \omega$ (including $\lambda = 0$), and
\item  the set of $(\rho_1, \rho_2)$ for which $\widetilde {\bf M}$ has a double real eigenvalue. 
\end{itemize}
Knowledge of these boundary sets, together with knowledge of the spectrum of the unperturbed operator ${\bf M}$, would therefore enable us to study stability in the entire plane. 

In passing we note that, while the unperturbed operator ${\bf M}$ is self-adjoint in some of the examples presented here, self-adjointness is not strictly necessary. What \textit{is} necessary is that the spectrum of the unperturbed operator ${\bf M}$ is known (at least qualitatively), so that we have a baseline with which to compare the perturbed operator (much as a constant of integration fixes a particular solution of a differential equation). In each example studied here, the unperturbed operator ${\bf M}$ will have a purely real spectrum.

Our main result in this paper will be to identify -- and gave a recipe for computing -- a set of geometric quantities associated with the spectrum of a rank-two perturbation of a well-known operator, which can be used to 
analyze the perturbed operator in the entire plane: this is done in \S \ref{sec:basic_calc}. We will then apply our technique to 
three specific problems which can be written in the form of Eqn. \eqref{eqn:Rank1}. The first is a model of a coupled brainstem-cerebellum neuronal network called the \textit{oculomotor integrator} (\S \ref{sec:numerics}); the second is a continuum version of that model, in which the (relatively numerous) brainstem neurons are replaced by a neural ``line" (\S \ref{sec:continuum}).  Finally, we analyze a stability problem that arises in a nonlocal reaction-diffusion equation \cite{RS92} (\S \ref{sec:RS_model}). In this last problem, we also use an intermediate result from \S 1 (the Aronszajn-Krein formula and its consequences) to prove a new theorem about stability of stationary solutions.

\section{Basic Calculations}  \label{sec:basic_calc}

We begin with some general dimension-counting arguments. The 
matrix\footnote{For purposes 
of exposition we will consider the case where $\widetilde{\bf M}$ is a matrix, but 
everything we say will apply equally to the case where $\widetilde{\bf M}$ is an 
operator of compact resolvent.} $\widetilde{\bf M}$ is assumed to 
be a real $N\times N$ matrix. Real non-symmetric matrices will generically have a real eigenvalue of 
multiplicity higher than one on a set of codimension one. In a two-parameter 
model such as we are considering here this codimension one set divides the 
parameter space into open sets having a constant number of real eigenvalues. 
As one crosses this set the number of real eigenvalues changes by (generically)
two. This motivates the following definition:

\begin{df}
We define the bifurcation curve to be the locus of points ${\mathcal V} = (\rho_1,\rho_2)$ for which  $\widetilde {\bf M}$ has a real eigenvalue of multiplicity 
two or higher.  
\end{df}

For matrix problems, of course, there exists an algebraic procedure for determining the values in the $(\rho_1,\rho_2)$ plane where the matrix has multiple eigenvalues. One can simply compute the discriminant (in $\lambda$)  
of the characteristic polynomial of the matrix $\widetilde {\bf M},$
\[
{\rm disc}_{\lambda}(\det(\widetilde{\bf M} - \lambda {\bf I})) = P(\rho_1,\rho_2),
\]
which gives a polynomial in the parameters $(\rho_1,\rho_2)$. The variety defined by the zero set of this polynomial 
\[
{\mathcal V} = \{ (\rho_1,\rho_2) | P(\rho_1,\rho_2) = 0 \}
\]
determines the bifurcation curve. Unfortunately, this computation isn't 
practical to carry out analytically for real problems: for a large 
matrix, $P(\rho_1,\rho_2)$ will be a polynomial of large degree and the zero 
set will be difficult to compute. For the case of operators, even very nice 
ones, it is not clear that the discriminant makes sense at all. 

Instead we use the fact that the 
perturbations are of finite rank to give an explicit 
rational or algebraic parameterization of the bifurcation curve.
We begin by stating a preliminary lemma, which is basically the Aronszajn-Krein 
formula for rank one perturbations:

\begin{lem}
Let $\widetilde{\MM}$ be an $N\times N$ matrix defined as in Equation (\ref{eqn:Rank1}). The characteristic polynomial of $\widetilde{\MM}$ 
\[
\widetilde D(\lambda) = \det(\widetilde{\MM} - \lambda{\bf I})
\]
takes the following form:
\begin{equation}
\det(\widetilde \MM - \lambda {\bf I})  = D(\lambda) + P_1(\lambda) \rho_1 + P_2(\lambda) \rho_2 + Q(\lambda) \rho_1 \rho_2
\label{eqn:disc}
\end{equation}
where $D(\lambda)= \det({\bf M} - \lambda {\bf I})$ 
is the determinant of the unperturbed problem, $P_i(\lambda)$ are polynomials 
of degree (at most) $(N-1)$ and $Q(\lambda)$ of degree (at most) $N-2$.
 In the case where $\vec g_1$ and $\vec g_2$ (or $\vec f_{1,2}$) are 
linearly dependent $Q(\lambda) = 0$. 
\end{lem}
\begin{proof} 
Due to the rank two nature of the perturbation, the characteristic polynomial 
can contain no powers of $\rho_1$ or $\rho_2$ above the first. 
The easiest way to see this
is via multilinear algebra. The determinant is clearly polynomial in 
$\lambda,\rho_1,\rho_2$. A general term of the form $\rho_1^j\rho_2^k$ comes from 
the wedge product of $j$ factors of $\vec g_1$, $k$ factors of $\vec g_2$  and 
($N-(j+k)$) columns from  $\MM- \lambda {\bf I}$. Any term with more than 
one factor of $\vec g_1$ and one factor of $\vec g_2$ must be zero. Hence the determinant is of the form given in Eqn. \eqref{eqn:disc}.

Finally, if $\vec g_1$ and $\vec g_2$ are linearly dependent, then any wedge product including both $\vec g_1$ and $\vec g_2$ will vanish, so that $Q(\lambda) \equiv 0$.

The explicit form of the polynomials $P_i(\lambda), Q(\lambda)$ is 
easy to compute from the above construction (a detailed derivation is provided in the Appendix). If we let 
$\cof^t \left(\MM - \lambda {\bf I} \right)$ denote the transpose 
cofactor matrix of ${\bf M}- \lambda {\bf I}$ then one 
has the following formulae
\begin{align}
 P_1(\lambda) &= \langle\vec g_1 ,\cof^t \left(\MM - \lambda {\bf I} \right) \vec f_1 \rangle
 \nonumber \\
 P_2(\lambda) &= \langle\vec g_2 ,\cof^t \left(\MM - \lambda {\bf I} \right) \vec f_2 \rangle
  \label{eqn:AKformula_explicit_Mat}\\
Q(\lambda) &= \frac{1}{\det(\MM - \lambda {\bf I})} \left\vert\begin{array}{cc} \langle\vec g_1 ,\cof^t \left(\MM - \lambda {\bf I} \right) \vec f_1 \rangle & \langle\vec g_1 ,\cof^t \left( \MM - \lambda {\bf I} \right) \vec f_2 \rangle\\
\langle\vec g_2,\cof^t \left( \MM - \lambda {\bf I} \right) \vec f_1 \rangle & \langle\vec g_2, \cof^t \left(\MM - \lambda {\bf I} \right) \vec f_2 \rangle  \end{array}\right\vert 
\nonumber
\end{align}
If in addition $\MM$ is self-adjoint, we have alternative formulae, in which the polynomial form of $Q(\lambda)$, etc., is even more evident:
\begin{align}
 P_1(\lambda) &= \sum _i \langle \vec f_1,\vec \phi_i\rangle \langle \vec g_1, \vec \phi_i \rangle\prod_{j\neq i}(\lambda_j-\lambda) 
 \nonumber \\
 P_2(\lambda) &= \sum _i \langle \vec f_2,\vec \phi_i\rangle \langle \vec g_2, \vec \phi_i \rangle\prod_{j\neq i}(\lambda_j-\lambda) 
 \label{eqn:AKformula_explicit_SelfAdj}\\
Q(\lambda) 
&= \sum_{i,j} \left(\langle \vec f_1, \vec \phi_i\rangle \langle \vec g_1, \vec \phi_i\rangle \langle \vec f_2, \vec \phi_j\rangle \langle \vec g_2, \vec \phi_j\rangle - 
\langle \vec f_1, \vec \phi_i\rangle \langle \vec g_2, \vec \phi_i\rangle \langle \vec f_2, \vec \phi_j\rangle \langle \vec g_1, \vec \phi_j\rangle\right) \prod_{k\neq i,j} (\lambda_k - \lambda)
\nonumber
\end{align}
where $\lambda_i$ and $\vec \phi_i$ are the eigenvalues and eigenvectors of the unperturbed matrix ${\bf M}$.
\end{proof}

It is often more convenient to divide Eq.~\eqref{eqn:disc} by $\det(\MM - \lambda {\bf I})$ to put the eigenvalue condition  in the form 
\begin{eqnarray}
0 & = & 1+ \rho_1\langle \vec g_1, (\MM - \lambda {\bf I})^{-1} \vec f_1 \rangle + \rho_2 \langle \vec g_2, (\MM - \lambda {\bf I})^{-1} \vec f_2 \rangle \nonumber \\
& + & \rho_1 \rho_2 \left(\langle \vec g_1,  (\MM - \lambda {\bf I})^{-1} \vec f_1 \rangle  \langle \vec g_2, (\MM - \lambda {\bf I})^{-1} \vec f_2 \rangle - \langle \vec g_1, (\MM - \lambda {\bf I})^{-1} \vec f_2 \rangle  \langle \vec g_2, (\MM - \lambda {\bf I})^{-1} \vec f_1\rangle\right);   \label{eqn:const_evalue_wInv}
\end{eqnarray}
or even more simply, when $Q(\lambda) = 0$,
\[
0 = 1 + \rho_1\langle \vec g_1, (\MM - \lambda {\bf I})^{-1} \vec f_1 \rangle + \rho_2 \langle \vec g_2, (\MM - \lambda {\bf I})^{-1} \vec f_2 \rangle.
\]
This form has the advantage that it is expressed in terms of resolvents (i.e. $R_{\lambda} \equiv (\MM-\lambda {\bf I})^{-1}$), which are defined very generally for operators, rather than determinants and 
cofactors, which are not.

One geometric way to interpret this characteristic polynomial is as defining 
a one parameter family of rational curves, the curves of constant eigenvalue. 
For each value of $\lambda$, Eqn. \eqref{eqn:disc} defines a curve in the $(\rho_1,\rho_2)$ plane 
along which 
$\lambda$ is an eigenvalue. For instance the matrix has zero as an eigenvalue 
 along the curve 
\begin{align*}
& D(0) + \rho_1 P_1(0) + \rho_2 P_2(0) + \rho_1 \rho_2 Q(0) = 0 \\
\Rightarrow \; & \rho_2 = -\frac{D(0) + \rho_1 P_1(0)}{ P_2(0) + \rho_1  Q(0)}
\end{align*} 
in the $(\rho_1,\rho_2)$ plane. In the special case where $Q(\lambda) \equiv 0$, Eqn. \eqref{eqn:disc} defines a one-parameter family of lines. 

Given a family of curves it is always 
fruitful to consider its \textit{envelope}, 
which is the simultaneous solution to 
\begin{align}
& D(\lambda) + \rho_1 P_1(\lambda) + \rho_2 P_2(\lambda) + \rho_1 \rho_2 Q(\lambda) = 0 \label{eqn:const_evalue}\\
& D'(\lambda) + \rho_1 P_1'(\lambda) + \rho_2 P_2'(\lambda) + \rho_1 \rho_2 Q'(\lambda) = 0. \label{eqn:envelope}
\end{align}
Whereas each curve in the family encodes information on the location 
of a particular eigenvalue, the envelope of the curve encodes information 
on eigenvalue coincidence: this is the content of the next lemma. 
To simplify notation, we use the wedge notation for Wronskians: $f \wedge g = f g' - f' g$. 

\begin{lem}  \label{lem:env=bif} 
The solutions to (\ref{eqn:const_evalue}, \ref{eqn:envelope}) give the bifurcation curve ${\mathcal V}$.
If $Q(\lambda)$ is not identically zero the bifurcation curve $\mathcal V$ is generically given by the union of the pair of parametric curves given by
 \begin{eqnarray}
\rho_1 &=& \frac{-(P_1\wedge P_2 + D\wedge Q) \pm \sqrt{\left(P_1\wedge P_2 - D
\wedge Q\right)^2 - 4 (D\wedge P_1)(P_2\wedge Q)}}{2(P_1\wedge Q)} \\
\rho_2 &=& \frac{P_1\wedge P_2 - D\wedge Q \mp \sqrt{\left(P_1\wedge P_2 - 
D\wedge Q\right)^2 - 4 (D\wedge P_1)(P_2\wedge Q)}}{2(P_2\wedge Q)}.
\end{eqnarray}
If $Q(\lambda)$ is identically zero then the bifurcation curve is generically 
given by the parametric curve 
\begin{eqnarray}
\rho_1 &=& -\frac{P_2\wedge D(\lambda)}{P_1\wedge P_2(\lambda)} \label{eqn:env1}
\\
\rho_2 &=& \frac{P_1\wedge D(\lambda)}{P_1\wedge P_2(\lambda)}  \label{eqn:env2}
 \end{eqnarray}
\end{lem}

\begin{proof}
The equivalence of the envelope and the discriminant is standard -- see for instance Bruce and Giblin\cite{BG} or Spivak\cite{spivak}. Generically one just has to 
solve Equations (\ref{eqn:const_evalue}, \ref{eqn:envelope}) for $\rho_1,\rho_2$. This is equivalent to 
solving a linear and a quadratic equation in the general case and a pair of linear equations in the special case $Q(\lambda)=0.$ 

For certain values of $\lambda$ the system (\ref{eqn:const_evalue}, \ref{eqn:envelope}) maybe be inconsistent, or consistent but underdetermined. Inconsistency indicates that this eigenvalue cannot be achieved by any choice of $\rho_{1,2}$; this will occur if
\[  \rank\left(\begin{array}{ccc} P_1 & P_2 & Q\\ P_1^
\prime & P_2^\prime & Q^\prime \end{array}\right) < \rank\left(\begin{array}{cccc}D & P_1 & P_2 & Q\\ D^\prime& P_1^
\prime & P_2^\prime & Q^\prime \end{array}\right).
\]

If the system is consistent but underdetermined, then there may be a curve in the $(\rho_1,\rho_2)$ plane along which $\lambda$ is a multiple eigenvalue. 
The system (\ref{eqn:const_evalue}, \ref{eqn:envelope}) is consistent and underdetermined for $\lambda$ if one of the following conditions ${\bf C1,C2,C3}$ hold: 
\begin{eqnarray}
&{\bf C1} ~~~~~ \rank\left(\begin{array}{cccc}D & P_1 & P_2 & Q\\ D^\prime& P_1^
\prime & P_2^\prime & Q^\prime \end{array}\right)& < 2    \label{eqn:checkC1}\\ 
&{\bf C2} ~~~~~ \left\{ \begin{array}{c} P_1\wedge Q(\lambda) \neq 0 \\ P_2\wedge Q(\lambda) = 0 \\ D\wedge P_1(\lambda)=0 \\ P_1\wedge P_2(\lambda) = 
D\wedge Q(\lambda) \end{array}\right. & \\ 
&{\bf C3} ~~~~~ \left\{ \begin{array}{c} P_2\wedge Q(\lambda) \neq 0 \\ P_1\wedge Q(\lambda)= 0 \\ D\wedge P_2(\lambda)=0 \\ P_2\wedge P_1(\lambda)=D\wedge Q(\lambda) \end{array}\right.  &
\end{eqnarray}
If any of these genericity conditions are satisfied
for some value of $\lambda$ then (generically) there exists a curve in the 
$(\rho_1,\rho_2)$ plane along which that value 
of $\lambda$ is a multiple eigenvalue. 
If they are not satisfied for any $\lambda$, then the bifurcation curve is equal to the envelope curve.
Note that one can always check whether 
or not a pair of polynomials have a common root by computing the 
resultant of the polynomials: one need {\em not} be able to explicitly factor 
the polynomials to test this condition. Thus the genericity condition is 
readily checkable.
\end{proof}

\begin{lem}
As $(\rho_1,\rho_2)$ are varied so as to cross the envelope the number of real eigenvalues 
generically changes by two.  
\end{lem}
\begin{proof}
We assume that most readers are familiar with this phenomenon from the
theory of first order quasilinear partial differential equations, where
the characteristic curves form a one parameter family of curves and
the envelope (caustic) marks the transition between regions which are
(typically) singly and triply covered by characteristics, but we give
a short proof.

The envelope curve is defined by the simultaneous solution to 
\begin{align*}
&F(\rho_1,\rho_2,\lambda)=0\\
&\frac{\partial F}{\partial \lambda} (\rho_1,\rho_2,\lambda) = 0.
\end{align*}
Assume that $(\rho_1^*,\rho_2^*)$ is a point along this curve with
corresponding eigenvalue $\lambda^*$, and that the gradient
$\nabla_{\rho_1,\rho_2} F(\rho_1^*,\rho_2^*,\lambda^*)$ and the second
derivative $\frac{\partial^2 F}{\partial
  \lambda^2}(\rho_1^*,\rho_2^*,\lambda^*)$ are
non-vanishing. Expanding in a neighborhood of this point, i.e. letting $\rho_1 =
\rho_1^* + \delta \rho_1, \rho_2 = \rho_2^* + \delta \rho_2$ and $\lambda =
\lambda^* + \delta \lambda$, we find the normal form 
\[
\frac12 \frac{\partial^2F}{\partial \lambda^2} (\delta \lambda)^2 +
\nabla_{\rho_1,\rho_2} F \cdot (\delta\rho_1,\delta\rho_2) = O(\delta \rho \delta \lambda, \delta\rho^2, \delta\lambda^3)
\] 
By the Weierstrauss preparation theorem, we can approximate solutions to $F(\rho_1, \rho_2, \lambda) = 0$ in a neighborhood of the point $(\rho_1^*,\rho_2^*)$ by fixing $(\delta \rho_1, \delta \rho_2)$ and solving for $\delta \lambda$; i.e. by solving the quadratic equation
\begin{eqnarray}
\frac12 \frac{\partial^2F}{\partial \lambda^2} (\delta \lambda)^2 + 
\nabla_{\rho_1,\rho_2} F \cdot (\delta\rho_1,\delta\rho_2)  & = & 0 \label{eqn:two_roots_nf}
\end{eqnarray}
for $\delta \lambda$. 
The tangent line to the envelope at $(\rho_1^*,\rho_2^*)$ is given by $\nabla_{\rho_1,\rho_2} F \cdot
 (\delta\rho_1,\delta\rho_2)=0$; on one side of this line, Eqn. \eqref{eqn:two_roots_nf} 
 has two distinct real roots $\pm \delta \lambda$, so each
 point lies on two constant eigenvalue curves, $\lambda^* \pm \delta \lambda$. 
 On the other side of the tangent line there are a pair of complex conjugate roots.
This general picture is illustrated in Figure (\ref{fig:bifurcation}), 
which shows a close-up of the envelope curve from an example to follow. 

For more details on the envelope see the text of Bruce and
Giblin\cite{BG}. 
\end{proof}

\noindent
\textbf{Figure 1 here}\\

We would next like to consider the possibility of eigenvalues of 
higher multiplicity. 
The envelope curve is, in general, a well-behaved curve and admits a 
parametrization by arc length. 
However this may fail at an isolated set of points in $(\rho_1,\rho_2)$. 
The next lemma says that (modulo some genericity assumptions) the following 
are all equivalent, and occur on a codimension 
two set (isolated points in the $(\rho_1,\rho_2)$ plane):
\begin{itemize}
\item Points in the $(\rho_1,\rho_2)$ plane where the model has a real eigenvalue
of multiplicity at least three.
\item Points where the tangent vector to the envelope curve vanishes.
\item Cusps in the envelope curve.
\end{itemize}

\begin{lem}
The vanishing of the tangent vector to the envelope curve at a point implies 
that $\widetilde {\bf M}$ has an eigenvalue of multiplicity (at least) three at that point.
The converse holds as long as
the following determinant is non-zero at the point in question:
\[
\left|\begin{array}{cc} P_1 + \rho_2 Q & P_2 + \rho_1 Q \\ P_1^\prime + \rho_2 Q
^\prime &P_2^\prime + \rho_1 Q^\prime \end{array}\right| \neq 0.
\]
Alternatively, the model has a triple eigenvalue if and only if either 
\begin{align*}
(P_1\wedge P_2 \wedge Q)(P_1 \wedge P_2 \wedge D) &= (D \wedge P_2 \wedge Q)(P_1
\wedge D \wedge Q) = 0\\
P_1\wedge P_2 \wedge Q &\neq 0  
\end{align*}
or 
\[
P_1\wedge P_2 \wedge Q = P_1 \wedge P_2 \wedge D = D \wedge P_2 \wedge Q=P_1\wedge D \wedge Q=0.
\]
\end{lem}
\begin{proof}
The conditions for an eigenvalue of multiplicity (at least) two are given 
by (\ref{eqn:const_evalue},\ref{eqn:envelope}). Differentiating each of these with respect to $\lambda$ 
gives the following equations 
for $\rho_1^\prime,\rho_2^\prime$:
\begin{align*}
\rho_1^\prime(P_1(\lambda) + \rho_2Q(\lambda)) +  \rho_2^\prime(P_2(\lambda) + \rho_1Q(\lambda)) &= 0 \\
\rho_1^\prime(P_1^\prime(\lambda) + \rho_2Q^\prime(\lambda)) +  \rho_2^\prime(P^\prime_2(\lambda) + \rho_1Q^\prime(\lambda)) &= -(D^{\prime\prime}(\lambda) + \rho_1 P^{\prime\prime}_1(\lambda) +  \rho_2P^{\prime\prime}_2(\lambda) + \rho_1 \rho_2 Q^{\prime\prime})
\end{align*}
The conditions for an eigenvalue of multiplicity at least three are given by 
 (\ref{eqn:const_evalue},\ref{eqn:envelope}) together with the condition 
\begin{equation}
(D^{\prime\prime}(\lambda) + \rho_1 P^{\prime\prime}_1(\lambda) +  \rho_2 P^{\prime\prime}_2(\lambda) + \rho_1\rho_2Q^{\prime\prime}) = 0. \label{eqn:mult3}
\end{equation}
Thus it is clear that $\rho_1^\prime=0, \rho_2^\prime=0$ implies that the eigenvalue 
is of multiplicity (at least) three. Further if 
\[
\left|\begin{array}{cc} P_1 + \rho_2 Q & P_2 + \rho_1 Q \\ P_1^\prime + \rho_2 Q
^\prime &P_2^\prime + \rho_1 Q^\prime \end{array}\right| \neq 0
\]
then the above system can be solved uniquely for $\rho_1^\prime,\rho_2^\prime$ and
the existence of an eigenvalue of multiplicity three implies  $\rho_1^\prime=0, 
\rho_2^\prime=0.$ 

A bit more algebra gives another characterization 
of points where the eigenvalue has multiplicity three or higher. Equations 
(\ref{eqn:const_evalue},\ref{eqn:envelope},\ref{eqn:mult3}) form a system 
of three equations in three unknowns $\rho_1,\rho_2,$ and $\rho_3=\rho_1\rho_2$.
 Solving these three equations for $(\rho_1,\rho_2,\rho_3)$ and imposing the 
consistency condition $\rho_3 = \rho_1\rho_2$ 
shows that one has a root of multiplicity three if and only if either   
\begin{align*}
\left(P_1 \wedge P_2 \wedge Q\right) \left(P_1 \wedge P_2 \wedge D\right) &=
 \left(D \wedge P_2 \wedge Q\right) \left(P_1 \wedge D \wedge Q\right) \\
 P_1\wedge P_2 \wedge Q &\neq 0
\end{align*}
{\em or} all of the Wronskians
\begin{equation*}
P_1 \wedge P_2 \wedge Q=P_1 \wedge P_2 \wedge D=D \wedge P_2 \wedge Q=P_1 \wedge
 D \wedge Q=0
\end{equation*}
vanish. 
The first possibility is typically of codimension two - it is
expected to occur at isolated values of $\lambda$ corresponding to 
isolated values of $(\rho_1,\rho_2)$. The second does not typically 
happen at all, since it requires the simultaneous vanishing of 
several polynomials. However in Example (1) this case occurs because it 
is forced by a symmetry 
of the model.
   
Finally recall that a simple zero of the tangent vector represents a cusp, and 
generically an eigenvalue of multiplicity at least three will have 
multiplicity exactly three, so typically cusps in the envelope 
are equivalent to triple eigenvalues. 
\end{proof}

The geometry of a bifurcation in the neighborhood of a 
triple eigenvalue is illustrated in Figure \ref{fig:cn_imp}B. 
In a neighborhood of this 
point there are three dominant eigenvalues which participate in the bifurcation.
The cusp of the envelope represents a transition between a 
bifurcation between the intermediate and the  smallest eigenvalue in the trio, and 
a bifurcation between the intermediate and the largest eigenvalue in the trio. 
Emerging from the cusp-point is a curve that represents 
an exchange of dominance phenomenon, with a complex conjugate pair of eigenvalues crossing a single real one.  

When considering questions of stability and the behavior of the dominant 
eigenvalue it is also important to understand the behavior of the 
complex eigenvalues. In particular one would like to understand the locus of 
points at which the matrix has purely imaginary eigenvalues, as this 
curve indicates where the model loses stability due to a Hopf bifurcation. 

\begin{df}
The Hopf curve is the locus of points in the $(\rho_1,\rho_2)$ plane where
$\widetilde {\bf M}$ has a pair of purely imaginary eigenvalues. Generically 
this curve is given parametrically by 
\begin{align}
{\rm Re}(D(i \omega))+ \rho_1 {\rm Re}(P_1(i \omega)) + \rho_2 {\rm Re}(P_2(i \omega)) + \rho_1\rho_2 {\rm Re}(Q(i\omega))&=0 \\
{\rm Im}(D(i \omega))+ \rho_1 {\rm Im}(P_1(i \omega)) + \rho_2 {\rm Im}(P_2(i \omega)) + \rho_1\rho_2 {\rm Im}(Q(i\omega))&=0,     
\end{align}
where ${\rm Re,Im}$ represent the real and imaginary parts respectively. 
The genericity conditions are the same as in Lemma (\ref{lem:env=bif}) 
with the replacement of the Wronskians $f\wedge g$ by the 
quantities ${\rm Re}(f (i \omega)){\rm Im}(g (i\omega)) - {\rm Re}(g (i\omega)){
\rm Im}(f (i\omega)).$ 
\end{df}

The Hopf curve, the envelope, and the 
$\lambda = 0$ eigenvalue curve will all intersect at a single point, the 
point at which there is a zero eigenvalue of higher multiplicity.
We now present an illustrative example. 

\begin{example}
We consider the following model: 
\begin{equation*}
\MM = \left(\begin{array}{cccc} -2 & -1 & 0 & -\rho_1 \\ -1 & -2 & -\rho_2 & 0 \\ \sqrt{2} & 1 & -2 & 0 \\ 1 & \sqrt{2} & 0 & -2 \\ \end{array} \right).
\end{equation*}
It is straightforward to compute that the characteristic polynomial of this
 matrix is given by 
\begin{equation*}
\begin{split}
\det(\MM - \lambda {\bf I}) &= (1+\lambda)(2+\lambda)^2(3+\lambda) + 
\left(\lambda^2 + (4 - \sqrt{2})\lambda + (4 - 2 \sqrt{2})\right)\rho_1 \\
&+ \left(\lambda^2 + (4 - \sqrt{2})\lambda + (4 - 2 \sqrt{2})\right)\rho_2 
-\rho_1\rho_2 \end{split}
\end{equation*}
The zero eigenvalue curve is given by
\begin{align*}
12 + (4&  - 2 \sqrt{2}) \rho_1 + (4 - 2 \sqrt{2}) \rho_2 - \rho_1 \rho_2 =0 \\
&\rho_1 = \frac{12 + (4  - 2 \sqrt{2}) \rho_2 }{\rho_2 - (4 - 2 \sqrt{2})}
\end{align*}
 
The bifurcation curve is given by the envelope 
\begin{align*}
\rho_1 &= -\left( \lambda+2 \right)\left( \lambda + 2 +\frac{\sqrt{2}}{2}\right) \pm \left(\lambda+2\right)\sqrt{2 \left( \lambda+\frac{3}{2}\right)\left(\lambda+\frac{5}{2}\right)}\\
\rho_2 &= -\left(\lambda+2 \right)\left(\lambda + 2 +\frac{\sqrt{2}}{2}\right) \mp \left(\lambda+2\right)\sqrt{2 \left(\lambda+\frac{3}{2}\right)\left(\lambda+\frac{5}{2}\right)}
\end{align*}
together with the singular piece $\rho_1=-\frac{1}{2} \cup \rho_2=-\frac{1}{2}$,
which is associated to the value $\lambda=-2+\frac{\sqrt{2}}{2}$, where the 
equations defining the envelope fail to have full rank. 
The envelope and the singular piece of the bifurcation curve meet tangentially at $(\rho_2=-\frac{1}{2},\rho_1=-\frac{3}{2})$ and 
$(\rho_2=-\frac{1}{2},\rho_1=-\frac{3}{2})$. Because of the symmetry we have $P_1=P_2$ 
and thus $P_1\wedge P_2 \wedge Q \equiv 0$, so the condition for a 
triple eigenvalue reduces to simultaneous vanishing of $P_1\wedge D\wedge Q$ and
 $D \wedge P_2 \wedge Q.$ Note that since $P_1=P_2$ these are not independent ---
 $P_1\wedge D\wedge Q=-D \wedge P_2 \wedge Q.$ Calculating we find that the triple eigenvalue condition becomes 
\[
P_1\wedge D\wedge Q = -16 \lambda^3+(12 \sqrt{2}-96) \lambda^2+(48 \sqrt{2}-192)
 \lambda+(46 \sqrt{2}-128) = 0
\]
This cubic has three real roots: a double root at 
$\lambda=-2+\frac{\sqrt{2}}{2}$ and a
simple root at $\lambda = -(\frac{1}{2} + \frac{\sqrt{2}}{4})\approx-2.35.$ 
The envelope is not defined for $\lambda \in (-\frac{5}{2},-\frac{3}
{2}),$ so the root at $\lambda = -(\frac{1}{2} + \frac{\sqrt{2}}{4})$ 
does not correspond to a real multiple eigenvalue. 
Thus the only real eigenvalue of 
multiplicity higher than two is $\lambda=-2+\frac{\sqrt{2}}{2}.$ Since this 
eigenvalue is associated to the singular piece of the bifurcation curve we can 
potentially have many points where this is a triple eigenvalue. Along the
curve $\rho_1=-\frac{1}{2}$ the eigenvalues are 
\[
\lambda = -2 + \frac{\sqrt{2}}{2},-2 + \frac{\sqrt{2}}{2},-2 -\frac{\sqrt{2}}{2}
 \pm 
\frac{\sqrt{2-4 \rho_2}}{2}.
\]
So the only triple eigenvalue is at $\rho_2=-\frac{3}{2}$, the point of 
intersection with the envelope curve. A similar calculation holds along  $\rho_2
=-\frac{1}{2}.$

The Hopf curve is given parametrically by 
\begin{eqnarray}
&\rho_1 = (4+\sqrt{2}) \frac{(4 \omega^2-14) \pm \sqrt{30(18-8\sqrt{2} 
+ (1-2\sqrt{2})\omega^2 + \omega^4)}}{14} \\
&\rho_2 = (4+\sqrt{2}) \frac{(4 \omega^2-14) \mp \sqrt{30(18-8\sqrt{2} + 
(1-2\sqrt{2})\omega^2 + \omega^4)}}{14}, 
\end{eqnarray}
where, as always, the signs are {\em not} independent. Note that the argument
of the square root is strictly positive, so there exists purely imaginary eigenvalues corresponding to oscillations of any desired frequency.
The Hopf curve, the envelope, and the zero eigenvalue line all 
meet at the points $(\rho_1=(4+\sqrt{2})(-1 \pm \frac{\sqrt{30(18-8\sqrt{2})}}
{14})= 4+\sqrt{2} \pm \sqrt{30},\rho_2=
(4+\sqrt{2})(-1 \mp \frac{\sqrt{30(18-8\sqrt{2})}}{14})=-(4+\sqrt{2}) \mp 
\sqrt{30}.$
The most interesting region of the stability diagram is 
depicted in Figure (\ref{fig:StabDiag}). The zero eigenvalue 
curve is depicted in dashed red, the envelope in blue
(including a dot at the origin), the singular piece of 
the bifurcation curve in dot-dashed magenta, and the Hopf curve in solid 
dotted green. 

From this information it is easy to derive the stability diagram. 
At the origin the eigenvalues are $\lambda=-3,\lambda=-2,\lambda=-2,\lambda=-1$.
Since there is a degenerate eigenvalue one needs to do a local perturbation 
analysis 
near $\lambda=-2,\rho_1=0,\rho_2=0$ to determine if in the neighborhood of this 
point one has a real pair of eigenvalues or a complex conjugate 
pair. Letting $\lambda = -2 + \delta$ shows that near this point one has 
\[
\det({\bf M} - \lambda {\bf I}) = -\delta ^2 - \sqrt{2} \delta \rho_1 - \delta \sqrt{2} \rho_2 - \rho_1 \rho_2 + O(3), 
\]
where $O(3)$ denotes terms of order three or higher in $\delta,\rho_i.$
The discriminant of the above is $(\sqrt{2}\rho_1+\sqrt{2}\rho_2)^2 - 4\rho_1\rho_2 = 2\rho_1^2 + 2 \rho_2^2>0,$ indicating that in a neighborhood of the 
origin the double eigenvalue splits into a real (distinct) pair of eigenvalues. 
 Thus in the region containing the origin and bounded by the singular pieces 
of the bifurcation curve and the upper branch of the envelope (labelled ${\bf A}
$) 
there 
are four real eigenvalues in the left half-plane. As $\rho_2$ is decreased 
so as to cross the line $\rho_2=-\frac{1}{2}$ the first bifurcation 
occurs. Since this line corresponds to eigenvalue $\lambda=-2+\sqrt{2}$ and 
$-2 < -2+\frac{\sqrt{2}}{2} <-1$ 
the bifurcation consists of the two dominant real eigenvalues bifurcating to 
a complex conjugate pair. Thus in region ${\bf B}$ we have two real eigenvalues
and two complex eigenvalues, all in the left half-plane.  
As one leaves region {\bf B} across the Hopf curve into the region 
labelled ${\bf E}$ the complex conjugate pair moves into the right half-plane, 
giving two complex eigenvalues in the right half-plane and two real 
eigenvalues in the left half-plane. Proceeding in
this fashion the stability diagram can be labelled as follows:
\begin{itemize}
\item Region A: Four real eigenvalues in the left half-plane. 
\item Region B: Two real and two complex eigenvalues in the left half-plane. 
\item Region C: Four complex eigenvalues in the left half-plane. 
\item Region D: Two complex eigenvalues in right half-plane, two complex 
eigenvalues in the left half-plane.
\item Region E: Two complex eigenvalues in the right half-plane, two real 
eigenvalues in the left half-plane.
\item Region F: One real eigenvalue in the right half-plane, three real eigenvalues in the left half-plane.
\item Region G: One real eigenvalue in the right half-plane, one real and two complex eigenvalues in the left half-plane.  
\end{itemize}
\end{example}
Additionally there is a narrow region between the regions labelled
{\bf E} and {\bf F} (to the left of $\rho_2= -(4+\sqrt{2}) - \sqrt{30}\approx -1
0.9$ 
above the zero eigenvalue curve and below the envelope curve) where there 
are two real eigenvalues in the left half-plane and two real eigenvalues 
in the right half-plane. This region is not labelled since it is not 
visible on this scale.
 
One feature which we have not labelled are points where 
a real eigenvalue and a complex conjugate pair all have the same 
real part, corresponding to points where a real eigenvalue and a 
complex conjugate pair exchange dominance. Although it is easy to write down 
an implicit equation satisfied by these curves, it is generally 
not possible to find an explicit formula as for the 
Hopf curve, envelope, etc. 

Since the characteristic polynomial of this problem is of order four it 
would in principle be possible to extract the above information directly from 
the solution formula for the quartic.  In practice it would be 
exceedingly difficult to recover such detailed information. In
the next section we will consider a model that arises from a differential equation of order \textit{eight}; in this 
situation it is no longer possible even in principle to state a direct formula for the roots. \\

\noindent
\textbf{Figure 2 here}\\

In this section, we identified -- and gave a recipe for computing -- a set of geometric quantities associated with the spectrum of a rank-two perturbation of a well-known operator. 
Because some of these quantities (specifically the bifurcation curve, the $\lambda=0$ eigenvalue curve, and the Hopf curve) form curves which together partition the $(\rho_1, \rho_2)$ plane into open sets with qualitatively similar behavior, this information allows us to analyze the perturbed operator in the entire plane.
We now apply this procedure to three specific problems which can be written in the form of Eqn. \eqref{eqn:Rank1}: a model of a coupled brainstem-cerebellum neuronal network called the \textit{oculomotor integrator} (\S \ref{sec:numerics}); a continuum version of that model, in which the (relatively numerous) brainstem neurons are replaced by a neural ``line" (\S \ref{sec:continuum}); and a stability problem that arises in a nonlocal reaction-diffusion equation \cite{RS92} (\S \ref{sec:RS_model}).

\section{Example: a model for the oculomotor integrator} \label{sec:numerics}

The \textit{oculomotor integrator} is a neural network that 
is essential for eye movement control. This network holds your gaze steady despite transient body motions, by using cues it receives from oculomotor subsystems such as the vestibular system (which processes sensory input from the semi-circular canals of the inner ear) \cite{R89,r89b}.
Here, we present a model for the oculomotor integrator
which can either simulate normal 
integrator function, or a common eye movement disorder --- congenital or \textit{infantile nystagmus} (IN) --- with a few small changes.
We are able to make this determination by analysis of 
the envelope, constant eigenvalue curves and Hopf curve 
as described in \S \ref{sec:basic_calc}. Some of this analysis was presented in \cite{bba08}, which focused on using the model to simulate various eye movements associated with IN; here, we focus on analyzing the full phase space generated by the low-rank perturbations.

We first summarize how the integrator works. 
The neurons that compose the integrator are located mainly in a brainstem region known as the \textit{vestibular nucleus}. 
These neurons must integrate \textit{velocity} signals into a desired \textit{position}; thus, they perform the operation of integration with respect to time (or temporal integration). 
Single neurons produce a small amount of temporal integration, in that
a firing rate increase due to a transient input will decay with a time constant of about 5 ms. The observed integrator time constant is closer to 20 seconds (s), so the integrator network must lengthen the single-neuron time constant by about 4000 times \cite{r89b};
for a linear system, this corresponds to having an eigenvalue near zero \cite{BarreiroNN09}. Furthermore, the integrator should have the right \textit{gain} in response to velocity signals; the ratio of its response to the input should be appropriate. It should also be \textit{plastic}; i.e. it should be able to adjust, if injury or some other change occurs (for example, you adjust your oculomotor integrator gain when you get a new pair of glasses) \cite{BM85,GM76a,GM76b,tsrz94}.

In order for the oculomotor integrator to function, the vestibular nucleus must be connected to a second major brain region, the \textit{cerebellum}\cite{Robinson74,R76,b88}. These connections are essential both for normal operation and for plasticity \cite{zybg81,cgrst90,rambold02,NK03}. Neuroanatomical research has shown that these connections are asymmetric in an
important way: 
while the connections from the vestibular nucleus to the cerebellum are numerous and excitatory, the feedback connections from the cerebellum are sparse and are inhibitory \cite{egv90,lfsc85,lfcsl85,Sekirnjak_etal_2003,tg92,BV00}. 
Anastasio and Gad \cite{ag07} showed previously that this asymmetry permits the cerebellum to sensitively control both the time constant and the gain of the oculomotor integrator, despite (or perhaps because) the projections from the cerebellum back to the vestibular nuclei are so sparse. 
\emph{These sparse cerebellar-to-vestibular connections, which can plastically change their strength, are precisely the low-rank perturbations we analyze in detail here.}

Anastasio and Gad \cite{ag07} proposed a linear differential equation 
model that combined a 
vestibular network with sparse, asymmetric feedback connections from cerebellar \textit{Purkinje cells}, i.e., the evolution of the system is given by:
\begin{equation}
\frac{d\vec v}{dt} = \widetilde{\MM} \vec v + s(t) \vec b ~~~~~~~~\vec v(0)=\vec 0,  \label{eqn:diffeq}
\end{equation}
where $\vec v$ represents the response of the system 
(vestibular neurons and Purkinje cells
together), $\widetilde{\MM}$ represents a matrix of connections, $s(t)$ is a 
scalar input (velocity) signal to the integrator, and $\vec b$ 
a fixed vector representing the pattern in which the vestibular neurons 
receive the input signal. 

The connection matrix $\widetilde{\MM}$ proposed by \cite{ag07} was:
\begin{eqnarray}
\widetilde{\MM} & = & \alpha \left(\begin{array}{ccc} {\bf T} & -\rho_1 \vec u_1 & - \rho_2 \vec u_2  \label{eqn:matrix_AGmodel} \\
\vec w_1^t & -1 & 0 \\ \vec w_2^t & 0 & -1 \end{array} \right)  
\end{eqnarray}
The vestibular neuron-to-Purkinje cell coupling vectors are given by the $\vec w_i$;
the Purkinje cell-to vestibular neuron 
coupling vectors are given by $(\vec u_1)_j = \delta_{j,k_1}, (\vec u_2)_j = \delta_{j,k_2}$, where $\delta_{j,k}$ is the Kronecker delta. 

We state the remaining details for completeness; for a full explanation of the biophysical motivation, please see \cite{ag07,bba08}. The matrix ${\bf T}$ of effective connections 
between vestibular neurons is given by 
\begin{equation}
{\bf T} = \left(\begin{array} {cccccc} -1+\beta & \beta & 0 & 0& \cdots & 0 \\
\beta & -1+\beta & \beta & 0 & \cdots & 0 \\
\vdots & \vdots & \vdots & \vdots  & \ddots & \vdots \\
0 & 0 & 0 & 0 & \beta & -1 + \beta  \end{array}\right).
\label{eqn:model2}
\end{equation}
 The parameter $\alpha$ is set to $200\, {\rm s}^{-1}$ (corresponding to a typical $5\, {\rm ms}$ membrane time constant); the parameter 
$\beta$ is fixed so that the vestibular sub-network has a time constant of $0.2 \, {\rm s}$ in the absence of cerebellar interaction ($\rho_1=\rho_2=0$).  
The largest eigenvalue of ${\bf T}$ is given by $\lambda_1/\alpha  = -1 + \beta(1 + 2 \cos(\frac{\pi}{N+1}))$, where $N$ is the number of vestibular cells (on each side of the network); therefore we choose $\beta = \frac{\lambda_1/\alpha + 1}{1 + 2 \cos \left( \frac{\pi}{N+1} \right)}$, where $\lambda_1 = 5\, {\rm s}^{-1}$.

As we have noted, we will treat the Purkinje-to-vestibular feedback connections as our low rank perturbations; to relate this model to Eq.~\eqref{eqn:Rank1}, take
\begin{eqnarray}
\MM & = & \alpha \left(\begin{array}{ccc} {\bf T} & \vec 0 & \vec 0  \label{eqn:matrix_AGmodel_explicit_decomp} \\
\vec w_1^t & -1 & 0 \\ \vec w_2^t & 0 & -1 \end{array} \right)  ; \qquad \vec f_i = -\rho_i \left[ \begin{array}{c}\vec u_i\\0\\0\end{array} \right], \; i=1,2; \qquad (\vec g_i)_j =  \delta_{j,N+i}, \; i=1,2
\end{eqnarray}
where $N$ is the number of vestibular neurons (therefore $\MM$ is $(N+2) \times (N+2)$, and $\vec f_i, \vec g_i$ are in $\RM^{N+2}$). 
We will assume that the output of the integrator is a linear readout of the vestibular neuron responses, which for simplicity we take to be equal to $\vec b$: i.e. $\langle\vec b,\vec v(t)\rangle$.  Defining the eigenvectors $\vec e_i$
and adjoint eigenvectors $\vec f_i$ respectively by 
\begin{eqnarray*}
\widetilde{\MM} \vec e_i &=& \lambda_i \vec e_i \\
\widetilde{\MM}^t\vec f  &=& \lambda_i \vec f_i 
\end{eqnarray*}
the linear readout at time $t$ (assuming $\widetilde{\MM}$ diagonalizable)
is given by 
\[
\langle\vec b,\vec v(t)\rangle = \sum_i  \frac{ \langle \vec b, \vec e_i \rangle \langle\vec f_i,\vec b\rangle}{\langle\vec f_i, \vec e_i\rangle} \int_0^t e^{\lambda_i(t-t')}s(t') dt'. 
\]
When there is a separation of time scales (${\rm Re} (\lambda_2) \ll {\rm Re} (\lambda_1)$), the response is largely due to the dominant eigenvalue; we define the \textit{gain}, $\gamma$, to be the ratio of this response, to the magnitude of the filtered input:
\begin{eqnarray}
\langle\vec b,\vec v(t)\rangle & \approx  & \frac{ \langle \vec b, \vec e_1 \rangle \langle\vec f_1, \vec b\rangle}{\langle\vec f_1, \vec e_1\rangle} \int_0^t  e^{\lambda_1(t-t')}s(t') dt'\\
& =  & \frac{ \langle \vec b, \vec e_1 \rangle \langle\vec f_1, \vec b \rangle}{\langle\vec f_1, \vec e_1\rangle \Vert \vec b \Vert^2}  \times \left[ \Vert \vec b \Vert^2 \int_0^t  e^{\lambda_1(t-t')}s(t') dt' \right]\\
 &= &\gamma  \left[ \Vert \vec b \Vert^2 \int_0^t  e^{\lambda_1(t-t')}s(t') dt' \right] \label{eqn:gain_def}
\end{eqnarray}
Thus, $\gamma$ captures how the circuit amplifies --- or suppresses --- the incoming signal.

Setting this ratio correctly allows the organism to respond appropriately to its environment. However, injury or normal growth can alter the sensory systems that supply inputs to the integrator, and thus create a mismatch between input and desired output. To counteract this, the system must change this gain to compensate. We will allow our network to change, by adjusting the Purkinje-to-vestibular weights $\rho_1$ and $\rho_2$; this is biologically plausible, since one of the major functions of the cerebellum is to regulate motor plasticity and learning. 

The asymmetry of the matrix $\widetilde{\MM}$ is crucial, to allow gain to be adjusted freely; for a symmetric (more generally \textit{normal}) matrix, $-1 \le \gamma \le 1$.  When can we get big changes in gain, with relatively small changes in $\rho_1$ and $\rho_2$? It would be ideal for the denominator of $\gamma$ to be near zero (assuming both $\vec f_1$, $\vec e_1$ are normalized to unit length): $\langle\vec f_1, \vec e_1\rangle \ll 1$. However, this corresponds to near-orthogonality of the right and left eigenvectors, which occurs near a double eigenvalue (perfect orthogonality, $\langle \vec f_1, \vec e_1 \rangle = 0$, can occur only when there is a degenerate double eigenvalue).

We now show an example of an integrator that can perform normal integration
with arbitrary adjustment of gain: let ${\bf T}$ be $6 \times 6$ and choose
\begin{eqnarray}
\vec u_1 & = & \vec {\bf e}_1, \qquad \vec u_2 =  \vec {\bf e}_3   \nonumber\\
\vec w_1^t &= &  \left[ \begin{array}{cccccc} -1 & 1 & -1 & 0 & -1 & 0 \end{array} \right]   \label{eqn:matrix_norm}\\
\vec w_2^t &= &  \left[ \begin{array}{cccccc} 1 & -1 & 1 & 1 & 0 & 0 \end{array} \right]  \nonumber
\end{eqnarray}
where $\vec {\bf e}_j$ is the $j$th identity vector (in this case $\vec {\bf e}_j \in \RM^6$). 


\noindent
\textbf{Figure 3 here}\\

With this choice, $\widetilde{\MM}$ is the reduced matrix of a model with 6 vestibular neurons 
and 2 Purkinje cells on each side of the bilaterally symmetric
network. 
The only parameters that may vary are $\rho_1$ and $\rho_2$, the strengths
of the Purkinje-to-vestibular connections. In order to
fix a certain time constant, $\rho_2$ and $\rho_1$ should be constrained 
to lie on the appropriate constant eigenvalue curve.
The biologically appropriate time constant is in the neighborhood of 
$20\, {\rm s}$, so the appropriate eigenvalue  is 
$\lambda = -\frac{1}{20}.$  From the results of \S \ref{sec:basic_calc} we know
that this holds along the curve 
\[
Q(-\frac{1}{20})\rho_1\rho_2 + P_2(-\frac{1}{20}) \rho_2 + P_1(-\frac{1}{20}) \rho_1 + D(-\frac{1}{20})
 = 0
\]
or (approximately)
\begin{equation}
\rho_1 = \frac{0.137 + 2.536 \rho_2}{1+.371\rho_2}
\label{eqn:consteig_norm_tau}
\end{equation}
This constant eigenvalue curve is tangent to the envelope at the simultaneous 
solution of 
\begin{align*}
0.137 - \rho_1 + 2.536\rho_2 - 0.371 \rho_1 \rho_2 &= 0\\
0.577 - \rho_1 + 2.405\rho_2 - 0.473 \rho_1 \rho_2 &= 0.
\end{align*}
(Note: we have divided each equation through by a constant.) The biologically 
important root of the above pair of equations is the one in the first quadrant,
$(\rho_2,\rho_1)=(1.22,2.23):$ negative values of $\rho$ would correspond 
to an \textit{excitatory} Purkinje-to-vestibular connection, which is not known to 
occur. 

The basic picture of integrator operation is as follows: let us suppose that 
$\rho_2$ is allowed to vary but that $\rho_1$ is given by Eqn. \eqref{eqn:consteig_norm_tau}, so 
that $\lambda=-\frac{1}{20}$ is always an eigenvalue. 
As $\rho_2$ is increased from zero the dominant eigenvalue is fixed at 
  $\lambda=-\frac{1}{20}$ and the 
(in this case real) subdominant eigenvalue increases. 
Because we are nearing the envelope curve (and thus a degenerate double eigenvalue $\lambda_2 = \lambda_1 = -\frac{1}{20}$),
we expect the gain to increase. At $\rho_2=1.22$,
where the constant eigenvalue curve is tangent to the envelope,  
there is a collision of eigenvalues and the dominant eigenvalue is degenerate. 
As $\rho_2$ is further increased the formerly subdominant eigenvalue 
is now dominant - it is real and larger than $-\frac{1}{20}$. 
Furthermore, we can compute the gain
 associated to the dominant mode
along the $\lambda=-\frac{1}{20}$ curve and find 
\begin{equation}
\gamma = \frac{\langle \vec f_1,\vec b\rangle \langle \vec b,\vec e_1\rangle}{\Vert\vec b\Vert^2\langle \vec f_1,\
\vec e_1\rangle} \approx \frac{0.05 (\rho_2 + 1.43)(\rho_2 + 1.86)}{(1.22-\rho_2)(1.65+\rho_2)}
\label{gain:ratfunc}
\end{equation}
Note that, as expected, the denominator of the gain diverges at $\rho_2=1.22,$
where the constant eigenvalue curve is tangent to the envelope and 
the $\lambda=-\frac{1}{20}$ eigenvalue is degenerate. 
 
In Figure \ref{fig:norm_imp}A we show the response of the network
to an impulsive forcing of the form $f(t) = \delta(t) (1,1,1,1,1,1,0,0)^t$  
as $\rho_2$, $\rho_1$ are varied to increase gain. (Note that the impulsive 
forcing is equivalent to free decay with a corresponding initial 
condition). Three responses are shown, with $\rho_2$ set to $0.65$,$0.955$, 
and $1.095$ 
respectively. For each value of $\rho_2$, $\rho_1$ is set so
that the network lies on the constant time constant 
curve $\tau = 20 s$. The gains predicted based on a single dominant  
mode ($\gamma$ in Eqn. \eqref{gain:ratfunc}) 
are $2.52,5.92$ and $12.88$. The corresponding measurements from the impulse response (given by dividing the maximum response by $\Vert\vec b\Vert^2$) yielded $2.54,5.87$ and $12.53$ respectively; so we see excellent agreement. 

Figure \ref{fig:norm_imp}B displays the interaction of the
two dominant eigenvalues of the network in the
vicinity of the current operating region. In order to increase gain,
the network must climb up the $\lambda = -0.05$ curve in the
vicinity of the double eigenvalue point. At the intersection of this
curve with the envelope, the integrating eigenvalue exchanges
dominance with another real eigenvalue, producing an unstable integrator.
Note that the larger two gain cases straddle the point where the 
$\lambda = -\frac{1}{20}{\rm s}^{-1}$ curve crosses the envelope, 
indicating  an eigenvalue bifurcation in a subdominant mode. 
In this case it is the mode(s) with the next largest real part. 
At about $\rho_2 \approx 1.02$ the subdominant complex conjugate pair 
collides at the real axis, and for $\rho_2$ above this value the first three 
most dominant eigenvalues are all real. As $\rho_2$ increases 
the subdominant eigenvalue increases until $\rho_2 \approx 1.22$ where 
there is an eigenvalue collision and exchange of dominance. 

Figure \ref{fig:norm_imp}B
also uses letter labels to show the character of the dominant eigenvalue.
Normal operation requires the network to remain in regions A or G. If the network is in
error, it may wander into a region where the two eigenvalues are
complex (F), or into the region where one or both eigenvalues are
in the right half-plane (B,C, D or E). In neither case is normal operation possible. 

\textit{Infantile nystagmus} (IN)
is a hereditary disorder characterized by involuntary, periodic eye movements. These movements (or waveforms) can be broadly classified into two forms, \textit{jerk} and \textit{pendular}.  In jerk waveforms, the eye moves outward from the central position with increasing speed until interrupted by a sudden saccade; a pendular waveform resembles a sinusoidal oscillation.  The presence of a jerk waveform suggests an unstable eigenvalue ($ \lambda > 0$); as noted in many integrator models, the need to maintain an eigenvalue very near zero implies that 
an unstable eigenvalue is a natural consequence of imprecision.

In our model, the ability to modulate gain is enhanced near the bifurcation curve; this suggests that the network is \textit{also} near a point where the dominant eigenvalue is complex, 
which would generate the sinusoidal oscillations characteristic of pendular nystagmus. 
For example, consider the network specified in Eqn. \eqref{eqn:matrix_cn}. This is very similar to  Eqn. \eqref{eqn:matrix_norm}; the only change is that the
vestibular-to-Purkinje input from a handful of neurons has been altered.
Here, as the cerebellum attempts to increase gain by adjusting 
$\rho_1$ and $\rho_2$, it enters an oscillating regime.
\begin{eqnarray}
\vec u_1 & = & \vec {\bf e}_1, \qquad \vec u_2  \; =  \; \vec {\bf e}_3   \nonumber\\
\vec w_1^t &= &  \left[ \begin{array}{cccccc} -1 & 1 & 0 & 0 & -1 & 0 \end{array} \right]   \label{eqn:matrix_cn}\\
\vec w_2^t &= &  \left[ \begin{array}{cccccc} 1 & -1 & 0 & 0 & 1 & 0 \end{array} \right]  \nonumber
\end{eqnarray}

\noindent
\textbf{Figure 4 here}\\

We see the impulse response of the network in Figure \ref{fig:cn_imp}A.
As in Fig. \ref{fig:norm_imp}A, $\rho_2$ and $\rho_1$ are varied so
as to remain along the $\lambda = -\frac{1}{20}{\rm s}^{-1}$ 
curve. As $\rho_2$ (and the gain)
increase, the network enters a regime where an oscillation is superimposed
on normal integration. Figure \ref{fig:cn_imp}B illustrates
why this behavior occurs. 
As we follow the $\lambda = -\frac{1}{20}{\rm s}^{-1}$ curve from left to right through Region A,
gain will increase. However,
the eigenvalue curve has an intersection with the Hopf curve
far to the left of its intersection with the envelope.
At this point the integrating eigenvalue exchanges dominance with
a pair of imaginary eigenvalues. Beyond this point, the response 
of the network contains both the normal integrating mode and a
superimposed oscillation (Region F).

\section{Example: a continuum model of the oculomotor integrator} \label{sec:continuum}
We next consider a continuum model
that can be derived from the model of Anastasio and Gad
by replacing the (relatively numerous) vestibular neurons with a continuum 
vestibular ``line'' denoted by $\psi(x)$, while the relatively few Purkinje 
cells remain 
discrete.  

Suppose the network has $N$ vestibular cells on each side of the integrator, arranged in a row of length $L$. 
Recall that the
vestibular interaction matrix ${\bf T}$ has nearest-neighbor structure; each row of ${\bf T}$ takes the form $\beta v_{j-1} + (-1 + \beta) v_j + \beta v_{j+1}$ (Eqn. \eqref{eqn:model2}). This will converge, in the continuum limit ($N \rightarrow \infty$), to a second derivative, because:
\begin{eqnarray*}
\left( \frac{L}{N} \right)^2 \psi_{xx} & = & \psi_{i+1} + \psi_{i-1} - 2 \psi_{i} + 
o\left( \left( \frac{L}{N} \right)^3 \right)
\end{eqnarray*}  
The sum over vestibular cells in the equation for the Purkinje cells (two final rows of Eqn. \eqref{eqn:matrix_AGmodel}) can similarly be replaced by an integral:
\[ \sum_{j}  (\vec w_1^t )_j v_j \rightarrow  \frac{1}{L/N} \int_0^{L} \psi(x) \phi_1(x) dx \]
where $\psi(j \Delta x) = v_j$ and $\phi_1(j \Delta x) = (\vec w_1^t)_j $.

This leads to the following integro-differential eigenvalue problem on $L_2[0,L] \times \RM^2$
\begin{align}
& \beta ( \Delta x )^2 \psi_{xx} + (-1+3 \beta) \psi - \rho_1 P_1 \delta(x-x_1)- \rho_2 P_2 \delta(x-x_
2)  = \lambda \psi,   \label{eqn:psi_cont_def_match} \\
  &\psi(0)=0,~~~~\psi(L)=0  \nonumber\\
&- P_i + \frac{1}{\Delta x} \int_0^L \psi(x)\phi_i(x) dx  = \lambda P_i , \quad i=1,2 \label{eqn:P1_cont_def_match}
\end{align}
where we have replaced $L/N$ with $\Delta x$.  \footnote{We have also scaled the entire problem by the common factor of $\alpha$.} 
The delta function coupling reflects the sparseness of the Purkinje-to-vestibular connections, with $x_{1,2}$ denoting the points on the 
vestibular line innervated by the Purkinje cells, and the 
functions $\phi_{1,2}(x)$ represent the density of vestibular 
to Purkinje connections. 

Eliminating $P_1$, $P_2$ from Eqn. \eqref{eqn:psi_cont_def_match} algebraically, we arrive at the single equation 
\begin{eqnarray}
\psi_{xx} + \frac{-1+3\beta}{\beta (\Delta x)^2} \psi - \frac{\rho_1\langle \psi,\phi_1\rangle}{\beta (\Delta x)^3 (\lambda+1)
} \delta(x-x_1)- \frac{\rho_2\langle \psi,\phi_2\rangle} {\beta (\Delta x)^3 (\lambda+1)
} \delta(x-x_2) = \frac{\lambda}{\beta (\Delta x)^2} \psi, ~~~~
\psi(0)=0=\psi(L)
\label{eqn:reduce_match}
\end{eqnarray}

This maps to our original problem, Eqn.~\eqref{eqn:Rank1}, as follows:
\begin{equation} 
\MM \psi = ( \beta (\Delta x)^2 + (-1 + 3 \beta))\psi; \qquad f_i = \frac{\delta(x-x_i)}{\Delta x (\lambda+1)}, \; i=1,2; \qquad g_i = \phi_i, \; i=1,2.  \label{eqn:S3_MapEqn1}
\end{equation}

This model can be solved in much the same way as the discrete model analyzed in \S \ref{sec:numerics}. To illustrate we take $L=1$ and the vestibular-to-Purkinje connections to be
$\phi_1(x)=\phi_2(x)=1.$  
Note that when $\rho_1=0,\rho_2=0$, the vestibular neurons decouple from the 
Purkinje cells, and the eigenvectors $\psi$ are given by trigonometric functions with the appropriate boundary conditions:
\begin{align}
\psi_n(x) & = \sin \left( n \pi x \right)  \label{eqn:evectors_unperturb}
\end{align}
after which $P_1$, $P_2$ are obtained by using Eqns. (\ref{eqn:P1_cont_def_match}):
\begin{align}
P_1 &=  P_2 =  \frac{N^2(1 - \cos(n \pi))}{n\pi \left( 3\beta - \beta \left(\frac{n \pi}{N}\right)^2 \right)} 
\end{align}
The corresponding eigenvalues are given by $\lambda_n = 
 - 1 + 3 \beta - \beta \frac{n^2 \pi^2}{N^2}$, together with $\lambda = - 1$, an eigenvalue of multiplicity two 
corresponding to the Purkinje cells. 
Note that the even modes ($n=2k$) do not actually excite a Purkinje cell 
response (i.e. $P_1 = P_2 = 0$), and thus the even modes 
are eigenfunctions of this problem for all values of $\rho_1,\rho_2$: these modes do not change under 
perturbation by the Purkinje cells. 

We now apply the techniques developed earlier in the paper to find the 
lines of constant eigenvalue. We will not reproduce the entire calculation, but 
merely note a few salient points. The first step is to act on Equation 
\eqref{eqn:psi_cont_def_match} 
with the appropriate resolvent (inverse) operator 
(here, $ R_{\lambda} = (\beta ( \Delta x )^2 \partial_{xx} +3 \beta - (1+\lambda))^{-1}$).
 The terms  
$R_{\lambda} \delta(x-x_i)$ are simply 
the Green's function for the operator 
$\beta ( \Delta x )^2 \partial_{xx} +3 \beta - (1+\lambda)$ acting on $L_2[0,L]$ with 
Dirichlet boundary conditions, which can be easily 
calculated -- see (for example) the text of Keener\cite{Keener}; we also supply some details in the Appendix. 

The main conclusion, is that the perturbed problem will have piecewise trigonometric eigenfunctions of the form $\sin(\omega x)$, etc.  (with appropriate continuity conditions; see Eqn. \eqref{eqn:psi_Greens}), where
\begin{eqnarray}
\omega^2 & = & \frac{-\lambda -1 + 3 \beta}{\beta (\Delta x)^2} \Rightarrow  \lambda = -1+3\beta - \beta (\Delta x)^2 \omega^2  \label{eqn:lambda_omega_rel}
\end{eqnarray}

We now fix the remaining constants in the model: $x_1=\frac{1}{3}$ and $x_2=\frac{1}{2}$, and perform the calculation thus described.
We find that the resulting envelope curves have asymptotes when $\omega$ is a multiple of $4 \pi$ and $6 \pi$; this is a consequence of the location of the Purkinje cell innervation (i.e. $x_1$ and $x_2$).

One issue that becomes apparent, is that as the network becomes large we are unable to find an integrating eigenvalue using Eqn \eqref{eqn:lambda_omega_rel} in this way;  observe that $\lambda$ is bounded above by $\lambda = -1+3\beta - \beta (\Delta x)^2 \omega^2 <= -1 + 3 \beta \approx -\frac{1}{40} + \frac{39}{120} \left( \frac{\pi}{N+1} \right)^2$, which $\rightarrow -\frac{1}{40}$ as $N \rightarrow \infty$.
However, the desired eigenvalue for a 20 s time constant (recall that we have scaled out the 5 ms single neuron time constant, $\alpha = 200$) is $\lambda = -0.05/200$, which is much closer to zero than $-1/40$; as $N$ increases, the integrating eigenvalue will become out of reach. One way to view this observation is that as ratio of vestibular to cerebellar (Purkinje) cells increases, sparse cerebellar innervation becomes less able to alter the time constant of the vestibular network. 

The alternative, is to consider piecewise exponential eigenfunctions; i.e.  
\begin{eqnarray}
\psi(x) & = & \left\{ \begin{array}{l l} A \sinh \omega x, & x < x_1 \nonumber \\
B \sinh \omega x + C \cosh \omega x, & x_1 < x < x_2 \label{eqn:psi_Greens_exp}\\
D \sinh(\omega (L-x)), & x > x_2 \nonumber
\end{array} \right.
\end{eqnarray}
where now 
\begin{eqnarray}
 \omega^2 & = &\frac{\lambda + 1 - 3 \beta}{\beta (\Delta x)^2} \Rightarrow \lambda = -1 + 3 \beta +  \beta (\Delta x)^2 \omega^2
 \label{eqn:lambda_omega_rel_sinh}
\end{eqnarray}
Although such functions cannot match the boundary conditions of the \textit{unperturbed} problem, the discontinuities in the perturbed problem bring them back into consideration. We can plot the corresponding bifurcation curve; it begins near $(\rho_2, \rho_1) \approx (-0.09, 0.09)$ (for $N=12$) and increases without bound into the second quadrant. In Fig. \ref{fig:continuous_envelope_finiteN}, we plot the phase plane for several values of $N$; $N=12$, 24, 50 and 100.  In each panel, we show several pieces of the bifurcation curve in different colors: $(0, 4\pi)$ (purple), $(4\pi, 6 \pi)$ (blue), $(6\pi, 8\pi)$ (cyan), $(8\pi, 12\pi)$ (green), and the curve for exponential eigenfunctions (yellow). We note that the structure is stable, for increasing $N$; however, the corresponding values of $\lambda$ are ``compressed" (see Eqn. \eqref{eqn:lambda_omega_rel}). This is to be expected; in the finite-dimensional problem, the eigenvalues cluster together as $N$ increases; therefore we expect increasing density of constant eigenvalue curves as $N$ increases. 

We now return to the question that originally motivated our analysis in \S 2; can this network behave as an integrator; i.e. can it achieve the correct time constant and gain, by adjusting its Purkinje-to-vestibular weights $\rho_1$, $\rho_2$?   

Recall that we characterized the performance of an integrator by two quantities; the dominant eigenvalue $\lambda_1$ and the gain $\gamma$. In \S 2 we established that in order to achieve high gain, the parameters $\rho_1, \rho_2$ must be set near the bifurcation point for the corresponding eigenvalue $\lambda_1$; this is where $\lambda_1$ has multiplicity two, and therefore (if degenerate) where the angle between left and right eigenvectors is near zero. Finally, recall that connections from the cerebellum are always \textit{inhibitory}; by the sign convention in Eqn.~\eqref{eqn:reduce_match}, this means that both $\rho_1, \rho_2$ must be positive. \textit{Therefore, a necessary condition for the continuum model to act as an integrator, is for the bifurcation point corresponding to $\lambda_1$ to lie in the first quadrant.}

This clearly does not occur in Fig. \ref{fig:continuous_envelope_finiteN}, where in every panel the yellow curve lies entirely in the second quadrant $\rho_2 < 0$, $\rho_1>0$. Therefore this network --- where $\phi_1 = \phi_2 = 1$, $x_{1,2} = \{ 1/3, 1/2\}$, and plasticity comes from manipulating the strength of the Purkinje-to-vestibular connections $\rho_1$, $\rho_2$ --- cannot regulate both time constant \textit{and} gain.  

We now take one further step in generality, by asking how the network behaves as the Purkinje-to-vestibular connection \textit{locations} --- $x_1$ and $x_2$ --- change.  This would reflect a different source of network plasticity. 
We will show that for the vestibular-to-Purkinje projection patterns chosen here,  the $\rho_1, \rho_2$ coordinates of the bifurcation point will always be of opposite sign, and therefore will never occur in the first quadrant.
\begin{lem}
Consider the continuum integrator model defined in Eq.~\eqref{eqn:reduce_match} with $L=1.$ Suppose $\phi_1 = \phi_2 = 1$; then for sufficiently large $N$, the bifurcation point corresponding to the integrating eigenvalue $\lambda_1$ will not be found in the first quadrant ($\rho_1, \rho_2 > 0$), for any $0< x_1, x_2 < 1$. 
\end{lem}
\begin{proof}
We assume that $N$ is large enough that the eigenvalue of interest 
requires an exponential eigenfunction; i.e. $\lambda_1 >  -1+3\beta - \beta (\Delta x)^2 \omega^2$. We begin by directly computing the required terms in the characteristic polynomial: 
\begin{eqnarray}
D(\omega) & = & 1\\
P_1(\omega) & = &  \langle \vec g_1, R_{\lambda} \vec f_1 \rangle  = \frac{-1 + \cosh(\omega x_1) - \sinh(\omega x_1) \tanh \left( \frac{\omega}{2} \right)}{\beta^2 (\Delta x)^3 \omega^2 (3 + \omega \Delta x)^2}\\
P_2(\omega) & = & \langle \vec g_2, R_{\lambda} \vec f_2 \rangle = \frac{-1 + \cosh(\omega x_2) - \sinh(\omega x_2) \tanh \left( \frac{\omega}{2} \right)}{\beta^2 (\Delta x)^3 \omega^2 (3 + \omega \Delta x)^2} \label{eqn:P2_forSinh}
\end{eqnarray}
(refer to Eq.~\eqref{eqn:const_evalue_wInv} for definitions and the Appendix for a more detailed calculation in terms of resolvent operators).  Note that
\[ P_1(\omega) = f(\omega, x_1), \qquad P_2(\omega) = f(\omega, x_2) \]
for the same function $f(\omega, x)$.

Then by Eqns. (\ref{eqn:env1}, \ref{eqn:env2}) the coordinates of the bifurcation point corresponding to any desired $\omega = \omega_0$ will be:
\begin{eqnarray*} 
\rho_1(\omega_0) & = & -\frac{\partial P_2}{\partial \omega}(\omega_0)  \times \left( \frac{1}{P_1\wedge P_2(\omega_0)} \right)\\
\rho_2(\omega_0) & = & \frac{\partial P_1}{\partial \omega}(\omega_0)  \times \left( \frac{1}{P_1\wedge P_2(\omega_0)} \right);
\end{eqnarray*}
consequently, 
\[ \frac{\rho_1(\omega_0)}{\rho_2(\omega_0)} = -\frac{\frac{\partial f}{\partial \omega}(\omega_0, x_2)}{\frac{\partial f}{\partial \omega}(\omega_0, x_1)}
\]
We will show that $\frac{\partial f}{\partial \omega}(\omega,x) > 0$ for $\omega > 0$ and $0 < x < 1$; therefore, the ratio of $\rho_1(\omega_0)$ and $\rho_2(\omega_0)$ must be negative. 

Differentiating the function $f(\omega, x)$ given in Eqn. \eqref{eqn:P2_forSinh}, we find that:
\begin{eqnarray} 
\frac{\partial f}{\partial \omega} (\omega, x ) & = & \left( 12+8 \Delta x^2 \omega^2 \right) \left( 1 + \cosh(\omega) - \cosh(\omega x) - \cosh(\omega(1 - x)) \right)/(1+\cosh(\omega)) \nonumber\\
&- & 2 \omega (3 + \Delta x^2 \omega^2) \left(  (1-x) \sinh(\omega x) + x \sinh(\omega (1-x)) \right) \label{eqn:bigF}
\end{eqnarray}
It may not be obvious what the sign of this function is (the first line is positive and the second negative), but (re-)defining $F_{\omega}(x) \equiv \frac{\partial f}{\partial \omega} (\omega, x )$, we will show that $F_{\omega}(x) \ge 0$ for any $\omega > 0$ and $x \in [0,1]$.  We do this by confirming that $F_{\omega}(x)$ is concave in $(0,1)$, and that $F_{\omega}(0) = F_{\omega}(1) = 0$. Therefore by the minimum principle for superharmonic functions, $ F_{\omega}(x) > 0$ on the interior $x \in (0,1)$.

To confirm concavity, we check that the second derivative in $x$ is negative:
\begin{eqnarray}
\frac{\partial^2 F_{\omega}}{\partial x^2} & = &  \frac{-2\Delta x^2 \omega \left[ \cosh(\omega x) + \cosh(\omega(1-x) \right] - (3 + \Delta x^2 \omega ^2)\left[ x \sinh(\omega(1-x)) + (1-x) \sinh(\omega x)\right]}{\beta^2 \Delta x^3 (3 + \Delta x^2 \omega^2)^2 (1+\cosh(\omega))}.
\end{eqnarray}
Checking that $F_{\omega}(0) =  F_{\omega}(1) = 0$ is a simple matter of substituting $x = 0,1$ into Eqn. \eqref{eqn:bigF}.

\end{proof}

In conclusion, we cannot make this network act as an integrator, by adjusting its Purkinje-to-vestibular connections. Instead, we would have to adjust the underlying vestibular-to-Purkinje projection pattern $\phi_1$. \\

%
\noindent
\textbf{Figure 5 here}\\



\section{Example: The Rubinstein-Sternberg model} \label{sec:RS_model}

In this section we give a third, quite different, application for the technique based on low-rank perturbations. 
Rubinstein and Sternberg\cite{RS92} introduced a nonlocal model for 
phase separation of the form
\begin{align*}
 u_t &= \Delta u + f(u) -\frac{1}{|\Omega|} \int_\Omega f(u) dx \qquad \qquad (x,t) \in \Omega \times (0,\infty) \\
&{\bf n} \cdot \nabla u = 0 \qquad \qquad \qquad x \in \partial \Omega.  
\end{align*}
We will compute the stability of a standing front type solution 
of this model, where the stability operator takes the form of a rank
one perturbation to a standard Sturm-Liouville operator. This model has been
analyzed by a number of authors, most notably Freitas\cite{F94,F94b,F95,F99}, and the
closely related problem of the coarsening rate for  the
Cahn-Hilliard equation has been analyzed in the classic paper of Bates
and Fife\cite{BF93}. The main goal of this example is to illustrate the utility
of treating the problem using the rank-one perturbation formula: while
the final result appears to be new several of the intermediary results
have analogues in the work of Freitas, and we will point these out where
germane. 

We begin with the equation   
\begin{equation}
u_t = u_{xx} + f(u)  -\frac{1}{2 L} \int_{-L}^L f(u) dx \qquad \qquad \qquad u_x(\pm L)=0.  \label{eqn:allen-cahn}
\end{equation}
This equation always admits a constant solution and for sufficiently large 
widths it admits front type solutions.  At $L = \frac{n \pi \sqrt{f'(0)}}{2}$ 
a bifurcation occurs giving rise to a solution containing $n$ fronts. In the 
absence 
of the non-local term these front solutions are (for $L$ finite) always 
unstable. The most unstable mode has a non-vanishing mean, so the instability 
is connected with non-conservation of mass. Rubinstein and  Sternberg 
introduced the non-local term as a Lagrange multiplier to enforce
mass conservation and remove this instability mechanism.

It is easy to see that, if $u$ represents a stationary solution to 
Eqn. \eqref{eqn:allen-cahn} then the linearized evolution equation is 
given by 
\begin{equation}
v_t = v_{xx} + f'(u)v  -\frac{1}{2 L} \int_{-L}^L f'(u) v dx \qquad \qquad \qquad v_x(\pm L)=0.  \label{eqn:lineara-c}
\end{equation}
Therefore the associated eigenvalue problem takes the form of a rank-one perturbation of a self-adjoint 
operator.   To explicitly relate Eqn. \eqref{eqn:lineara-c} to Eqn. \eqref{eqn:Rank1} we can define
\begin{equation} 
\MM v = v_{xx} + f'(u)v; \qquad f_1 = -\frac{1}{2L}; \qquad g_1 = f'(u); \qquad f_2=g_2=0.  \label{eqn:S4_MapEqn1}
\end{equation}

The problem with a bistable cubic nonlinearity, $f(u) = u - u^3$,
 \[
 u_t = u_{xx} + u - u^3  -\frac{1}{2 L} \int_{-L}^L (u - u^3) dx \qquad \qquad u_x(\pm L)=0  
\] 
is the simplest and most natural from a physical perspective, and can be 
analyzed rather explicitly. Assuming $L > \frac{\pi}{2}$ there is a front 
solution which can be expressed in terms of elliptic functions\footnote{For fixed period $L$ there is actually a one-parameter family of front solutions. 
The one given here has zero net mass and is the simplest. The rest
will be discussed later.}. After a 
simple rescaling ($u = (1 + k^2)^{-1} v$, $x = (1 + k^2)^{-1/2} y$, and $t = \frac{\sqrt{2} k}{\sqrt{1+k^2}}s$), the equation can be written in the form
 \begin{equation}
u_t = u_{xx} + (1+k^2)u - 2 k^2 u^3  -\frac{1}{2 K} \int_{-K(k)}^{K(k)} ((1+k^2) u - 2 k^2 u^3) dx \qquad \qquad u_x(\pm K(k))=0. \label{eqn:front}  
\end{equation}
 Here the quantity $k\in(0,1)$ denotes the elliptic modulus and $K(k)$ denotes 
the complete elliptic integral of the first kind  
\[
K(k)=\int_0^1\frac{dx}{\sqrt{(1-x^2)(1-k^2 x^2)}}\in(\frac{\pi}{2},\infty).
\]
 The elliptic modulus $k$ is determined from $L$ by the relation  
\[
\sqrt{1+k^2} K(k) = L.
\]

The unperturbed problem now becomes 
 \begin{equation}
u_t = u_{xx} + (1+k^2)u - 2 k^2 u^3  \qquad \qquad u_x(\pm K(k))=0;  \label{eqn:front_unperturbed}  
\end{equation} 
a well-known identity for elliptic functions states that a stationary solution to this equation is given by 
\[
u(x;k) = {\rm sn}(x,k)
\]
where ${\rm sn}$ is the Jacobi elliptic sinus function.  

Note that as the function ${\rm sn}(x,k)$ is odd the 
nonlocal term $\frac{1}{2 K} \int_{-K(k)}^{K(k)} ((1+k^2) u - 2 k^2 u^3) \, dx$ vanishes; thus, this function solves \textit{both} the perturbed and unperturbed problem (i.e. both Eqn. \eqref{eqn:front} and Eqn. \eqref{eqn:front_unperturbed}). However, the non-local term 
changes the stability problem; we will find that the stability properties of the two problems 
are not the same.

Linearizing around the elliptic function solution gives the 
following non-local evolution equation
\begin{align}
v_t &= v_{xx} + (1+k^2) v - 6 k^2 {\rm sn}^2(x,k) v - \frac{1}{2 K(k)}\int_{-K(k)}^{K(k)} (1 + k^2 - 6 k^2 {\rm sn}^2(x,k) ) v dx    \\
&= {\bf H} v -\frac{1}{2 K(k)} \int_{-K(k)}^{K(k)} (1 + k^2 - 6 k^2 {\rm sn}^2(x,k) ) v \, dx  =: \widetilde {\bf H} v,
\label{eqn:front_cubic_stability}
\end{align}
where we have identified $\widetilde {\bf H}$ as a perturbation from a ``simpler" operator, ${\bf H}$.
The unperturbed operator ${\bf H}$ is a two-gap Lam\'e operator, for which the 
spectral problem can be solved exactly \cite{MR0197830,MR3075381}. When subject to periodic boundary 
conditions on $[-2K(k),2K(k)]$ the largest five eigenvalues are simple
and are given by 
\begin{equation}
\begin{array}{ll}
\phi_0^{(N)}(x) = k^2 {\rm sn}^2(x,k) - \frac{1 + k^2 + a(k)}{3}, & \lambda_0^{(N)} = -(1+k^2 - 2 a(k)) \\
\phi_1^{(D)}(x) = {\rm cn}(x,k){\rm dn}(x,k), & \lambda_1^{(D)} = 0 \\
\phi_1^{(N)}(x) = {\rm sn}(x,k){\rm dn}(x,k), & \lambda_1^{(N)} = -3k^2 \\
\phi_2^{(D)}(x) = {\rm cn}(x,k){\rm sn}(x,k), & \lambda_2^{(D)} = -3 \\
\phi_2^{(N)}(x) = k^2 {\rm sn}^2(x,k) - \frac{1 + k^2 - a(k)}{3}, & \lambda_2^{(N)} = -(1+k^2 + 2 a(k))
\end{array}
\label{eqn:eigs_RS_unperturbed}
\end{equation}
with $a(k) = \sqrt{1 - k^2 + k^4}$. Here the superscript indicates whether the 
function satisfies a Neumann or a Dirichlet condition at $\pm K(k)$. Thus the 
unperturbed operator subject to Neumann boundary 
conditions has one positive eigenvalue $\lambda_0 = -(1+k^2 - 2 a(k))$ and the 
remainder of the eigenvalues on the negative real line. 

To summarize, the stability problem (Eqn. \eqref{eqn:front_cubic_stability}, plus boundary 
conditions $v_x (\pm K(k)) = 0$) takes the 
form of a rank-one perturbation of a self-adjoint problem:
\[
\widetilde {\bf H}v  = {\bf H}v + G \langle g,v\rangle  
\]
with $G = -\frac{1}{2 K(k)}{\bf 1}$ and $g = (1+k^2 - 6 k^2 {\rm sn}^2(x,k))$. 
While neither $G$ nor $g$ is an eigenvector of ${\bf H}$, both $G$ and $g$ lie in the span of the zeroth and second 
eigenfunctions $\phi_{0/2}^{(N)}(x) = k^2 {\rm sn}^2(x,k) - \frac{1}{3}(1 + k^2 \mp a(k))$. Since the eigenfunctions of a self-adjoint operator 
are orthogonal this implies that the rank-one piece, $G \langle g,v\rangle$ 
vanishes on the span of all the remaining eigenfunctions.   
This implies that the perturbed operator decomposes as a 
direct sum of two operators, one that is self-adjoint and 
negative definite and one that is rank-two: 
\begin{eqnarray}
\widetilde {\bf H} & = & \underbrace{{\widetilde {\bf H}|_{{\rm span}(\phi_0,\phi_2)}}}_{Rank-two} \oplus\underbrace{\widetilde {\bf H}|_{{\rm span}(\phi_0,\phi_2)^\bot}}_{negative~definite} = {\widetilde {\bf H}|_{{\rm span}(\phi_0,\phi_2)}} \oplus {\bf H}|_{{\rm span}(\phi_0,\phi_2)^\bot}.
\end{eqnarray} 
We emphasize that the perturbation term here is very special, in that $G$ and $g$ are actually given by linear combinations of 
just two of the eigenfunctions of the unperturbed operator. In the generic case (i.e. for a different nonlinearity $f(u)$) one expects $G$ and $g$ to have non-trivial projections 
onto all eigenfunctions. 

From the exact eigenvalues in Eqn. \eqref{eqn:eigs_RS_unperturbed} the second term satisfies 
the coercivity estimate
\[
\widetilde {\bf H}|_{{\rm span}(\phi_0,\phi_2)^\bot} \leq -3k^2 {\mathbb I}
\]
and the entire stability problem is reduced to understanding the two by two 
matrix eigenvalue problem defined by $\widetilde {\bf H}|_{{\rm span}(\phi_0,\phi_2)}.$
Furthermore, since the range of ${\widetilde {\bf H}}$ consists 
of mean zero functions it follows from the Fredholm alternative that $\widetilde {\bf H}|_{{\rm span}(\phi_0,\phi_2)}$
must have a zero eigenvalue. 

This zero eigenvalue is connected with mass conservation. There is a 
one-parameter family of solutions to this equation of fixed spatial period 
$L$, which can be thought of as being related to the total mass of the 
stationary solution. Here we have only considered the simplest solutions 
-- those with zero net mass -- but there are analogous expressions in 
terms of elliptic functions for the general solution. Since we have a 
one parameter family of solutions, by Noethers theorem there must 
be an element in the kernel of linearized operator corresponding to the 
generator of this family.

It is straightforward to compute the restriction of the linearized operator 
$\widetilde {\bf H}|_{{\rm span}(\phi_0,\phi_2)}$ in the (independent but non-orthogonal) 
basis $\{1,{\rm sn}^2(x,k)\}$ in terms of complete elliptic integrals of the 
first and the second kind. By using the following identities:
\begin{eqnarray*}
E(k) &= \int_0^1 \frac{\sqrt{1-k^2 x^2} dx}{\sqrt{1-x^2}} \\
\int_{-K(k)}^{K(k)} {\rm sn}^2(x,k) dx & = \frac{2 (K(k)-E(k)}{k^2} \\
\int_{-K(k)}^{K(k)} {\rm sn}^4(x,k) dx & = \frac{(4 + 2 k^2)K(k)-(4 + 4 k^2)E(k)}{3k^4} \\
\end{eqnarray*}
we derive the following expression for $\widetilde {\bf H}|_{{\rm span}{\phi_0,\phi_2}}$, the restriction of the operator to the span of the zeroth and second eigenfunctions, in the basis $\{1,{\rm sn}^2(x,k)\}:$
\[
\widetilde {\bf H}|_{{\rm span}{\phi_0,\phi_2}} = \left(\begin{array}{cc} 6\frac{K(k)-E(k)}{K(k)} &  \frac{3 (1+k^2)\left( K(k) - E(k)\right) }{k^2 K(k)}\\ - 6k^2  & -3(1 + k^2)\end{array}\right).
\]
The eigenvalues of this restriction are given by 
\begin{align}
& \lambda_0 = 0 \\
& \lambda_1 = \frac{(3-3k^2)K(k) - 6 E(k)}{K(k)}.
\end{align}
  
It is easy to check directly from the definition of the elliptic 
integrals that the quantity $(3-3k^2)K(k) - 6 E(k)$ is strictly
negative for $k \in [0,1)$. This follows, for instance, from the 
Taylor series representations 
\begin{align*}
&K(k) = \frac{\pi}{2} \left(1 + \sum_{i=1}^\infty \left(\frac{(2i-1)!!}{(2i)!!}\right)^2 k^{2i} \right)\\
& E(k) =\frac{\pi}{2}\left(1 -\sum_{i=1}^\infty \left(\frac{(2i-1)!!}{(2i)!!}\right)^2 \frac{k^{2i}}{2i-1}\right),
\end{align*}
from which it is easy to see that the quantity $(1-k^2) K(k) - 2 E(k)$ is even with only negative 
terms in its Taylor series, implying that it is strictly negative. 
Since the equation conserves mass it 
makes sense to consider only perturbations with zero net mass 
(those orthogonal to the kernel). In this case we have all 
eigenvalues strictly in the left half-plane and the 
stationary solution is non-linearly stable.

The problem for a general nonlinearity is slightly more complicated, since we do not have 
explicit formulae for the eigenvalues and eigenfunctions, but it can essentially be completely solved.
Assuming that the equation 
\[
u_t = u_{xx} + f(u) - \frac{1}{2L}\int_{-L}^{L} f(u) dx \qquad \qquad u_x(\pm L)=0
\]  
has a stationary solution $u(x)$ we can compute the linearized operator 
as 
\begin{eqnarray*}
\lambda v &= &v_{xx} + f'(u) v -\frac{1}{2L} \int_{-L}^{L} f'(u) v \, dx, \qquad \qquad v_x(\pm L)=0 \\
\Rightarrow \lambda v &= &{\bf H} v - \frac{1}{2L} \langle f'(u),v \rangle   {\bf 1}  = \widetilde {\bf H} v
\end{eqnarray*}
where ${\bf H}$ is the unperturbed operator: ${\bf H} v =  v_{xx} + f'(u) v$.
It is convenient to introduce a coupling constant $\rho$ controlling the strength 
of the perturbation, although we are mainly interested in the special case $\rho = 1$, 
and we thus consider a one-parameter family of eigenvalue problems 
\[{\bf H}_\rho v = {\bf H} v - \frac{\rho}{2L} \langle f'(u),v \rangle   {\bf 1} = \lambda v. 
\]
Applying the Aronszjan-Krein formula gives the following eigenvalue
condition:
\begin{eqnarray*}
& ({\bf H}-\lambda {\bf I})v =  \frac{\rho}{2L} \langle f'(u),v \rangle {\bf 1} \Rightarrow \\
& v =     \frac{\rho}{2L} ({\bf H}-\lambda {\bf I})^{-1}  {\bf 1} \langle f'(u),v \rangle \Rightarrow \\
& \langle f'(u),v \rangle =  \frac{\rho}{2L}\langle f'(u), ({\bf H}-\lambda {\bf I})^{-1}  {\bf 1} \rangle \langle f'(u),v \rangle. 
\end{eqnarray*}
Thus the spectrum again decomposes into two pieces. Any eigenvectors of ${\bf H}$ which happen to be orthogonal to $f'(u)$ ($\langle f'(u),v\rangle=0$) remain 
eigenvectors of the
perturbed problem. Eigenvectors which are not orthogonal to $f'(u)$
must satisfy the Aronszajn-Krein eigenvalue condition
\begin{equation}
1 = \frac{\rho}{2L}\langle f'(u), ({\bf H}-\lambda {\bf I})^{-1} {\bf 1} \rangle
\label{eqn:EigCond}
\end{equation}
(This condition was also identified by Freitas \cite{F94}, who referred to the 
eigenvalues which satisfy Eqn. \eqref{eqn:EigCond} as ``moving eigenvalues'').

Now note that one has the identity ${\bf H} {\bf 1} = f'(u)$ and thus we can write 
\begin{eqnarray*}
1 & = & \frac{\rho}{2L}\langle {\bf H} {\bf 1}, ({\bf H}-\lambda {\bf I})^{-1} {\bf 1} \rangle  \\
& = & \frac{\rho}{2L}\langle ({\bf H}-\lambda {\bf I}+\lambda{\bf I}){\bf 1},  ({\bf H}-\lambda {\bf I})^{-1} {\bf 1} \rangle  \\
& = & \rho + \frac{\rho}{2L}\langle {\bf 1} \lambda, ({\bf H}-\lambda {\bf I})^{-1} {\bf 1} \rangle \Rightarrow \\
 &\Rightarrow & \frac{\lambda}{2L}\langle {\bf 1},  ({\bf H}-\lambda {\bf I})^{-1} {\bf 1} \rangle - \frac{1-\rho}{\rho} = 0 \\
 & \Rightarrow &\frac{1}{2L}\sum \frac{\langle {\bf 1},v_i\rangle^2}{\lambda_i-\lambda} - \frac{1-\rho}{\rho \lambda} = h(\lambda) =0
\end{eqnarray*}
where we use the identity ${\bf 1} = \sum \langle {\bf 1}, v_i \rangle v_i$ and therefore that $({\bf H} - \lambda {\bf I})^{-1} {\bf 1} = \sum \frac{\langle {\bf 1},  v_i \rangle}{\lambda_i - \lambda}  v_i$; here, $\lambda_i$ and $ v_i$ are the eigenvalues and (orthogonal) eigenvectors of the unperturbed operator ${\bf H}$.
Notice that in the last expression $h(\lambda)$ is a Herglotz
function\footnote{Sometimes called a Nevanlinna or Nevanlinna-Pick
  function.} --- an analytic function that is real on the real axis and  
maps the open upper half-plane to itself --- for $\rho \in [0,1].$  It
is well known that the zeroes and poles of a Herglotz function are
real, implying that the eigenvalues of the Rubinstein-Sternberg model for 
$\rho \in [0,1]$ are real (see Simon\cite{Simon2} page 920 for a 
more detailed discussion of Herglotz functions).  
This observation is analogous to Lemma 5.2 in the paper of
Freitas\cite{F95}, where reality of the eigenvalues for $\rho \in [0,1]$
is established using a combination of identities derived from the
original equation. 
In later work Freitas \cite{F99} considers
a general rank one perturbation of a self-adjoint operator: $\widetilde
{\bf H} = {\bf H} + |a \rangle \langle b| $, and associates to each
eigenvalue $\lambda_i$ a signature given by ${\rm sgn} (\langle b,v_i \rangle \langle
v_i,a\rangle)$, where $v_i$ is the eigenfunction of the unperturbed
problem --- see, in particular, section 3 of \cite{F99}. The condition that
all of these signatures are the same is equivalent to requiring
that the Aronszajn-Krein function is Herglotz.       

This calculation shows that, despite the fact that the linearized operator is not self-adjoint, the spectrum is purely real for $\rho \in [0,1]$. 
Now let us assume that we understand the spectrum of ${\bf H}_0 = {\bf H}$, the unperturbed operator, in particular that we know $n_+({\bf H})$,
the number of positive eigenvalues of the unperturbed operator.  Let us consider doing a homotopy
in the parameter $\rho$. Since we have shown that the eigenvalues are real the only way that an eigenvalue can move from the 
right half-plane to the left half-plane (or vice-versa) is by passing through the origin. We can detect when this occurs by taking 
$\lambda=0$ in equation (\ref{eqn:EigCond}), which becomes
\[
1 = \frac{\rho}{2L} \langle f'(u), {\bf H}^{-1} {\bf 1}\rangle = \frac{\rho}{2L} \langle {\bf 1}, {\bf 1} \rangle = \rho  
\]
Here we have used the fact that ${\bf H} {\bf 1} = f'(u)$ so that ${\bf 1} = {\bf H}^{-1} f'(u)$. (We are also assuming that ${\bf H}$ is invertible. Minor changes are required in the case that ${\bf H}$ has a kernel; see Remark 1.) This calculation shows that the unique value of $\rho$ for 
which ${\bf H}_\rho$  has a zero eigenvalue is $\rho=1$.

We are now in a position to count the number of positive eigenvalues of $\widetilde {\bf H}$ using a continuation argument. The operator $\widetilde{\bf H}$ is bounded above and, by standard arguments, has a finite number of positive eigenvalues. We assume that the number of 
positive eigenvalues of the unperturbed operator is given by $n_+({\bf
  H})=k$: at $\rho=0$ there are $k$ positive eigenvalues and the remaining eigenvalues are 
negative. For $\rho\in(0,1)$ the kernel of ${\bf H}_\rho$ is empty, so no eigenvalues cross from the left half-line to the right. 
At $\rho=1$ there is an eigenvalue at $\lambda=0$, which either came from the left half-line or the right half-line. We can determine which by computing $\frac{d\lambda}{d\rho}$ and evaluating at $\lambda=0$. Doing so we find that 
\[
\frac{d\lambda}{d\rho}= -\frac{2L}{\langle f'(u), {\bf H}^{-2} {\bf 1}\rangle} = -\frac{2L}{\langle {\bf 1},  {\bf H}^{-1} {\bf 1}\rangle}.
\]
If $\langle {\bf 1}, {\bf H}^{-1} {\bf 1}\rangle >0$ the eigenvalue is moving from the positive half-line to the negative, while 
if $\langle {\bf 1}, {\bf H}^{-1} {\bf 1}\rangle <0$ it is moving from the negative to the positive half-line. 

Finally, the assumption that the
unperturbed operator ${\bf H}$ is invertible implies that the
perturbed operator
$\widetilde {\bf H}$ has at most a one dimensional kernel, since if
$\widetilde {\bf H}$ had a higher dimensional kernel one of the
eigenfunctions could be chosen to be orthogonal to $f'(u)$ and would
thus lie in the kernel of the unperturbed operator ${\bf H}$.  
This completes the proof of the following theorem:
\begin{thm}
Suppose that the unperturbed operator ${\bf H} = \partial_{xx} + f'(u)$ is non-singular and has $n_+({\bf  H})=k$ positive eigenvalues. The perturbed operator 
$\widetilde {\bf H}$ has a simple kernel and $n_+({\bf {\widetilde H}})$, the number of positive eigenvalues of the linearized operator, is given by 
\[
n_+({\bf {\widetilde H}}) = \left\{ \begin{array}{ll} k, & \qquad \langle {\bf 1} ,{\bf H}^{-1} {\bf 1}\rangle <0 \\
k-1, & \qquad \langle {\bf 1},{\bf H}^{-1} {\bf 1}\rangle >0 \end{array}\right.
\] 
Thus a necessary and sufficient condition for spectral stability is that $n_+({\bf H})=0$ (from which it follows that $\langle {\bf 1},{\bf H}^{-1} {\bf 1}\rangle <0$) or $n_+({\bf H})=1$ and $\langle {\bf 1},{\bf H}^{-1} {\bf 1}\rangle >0.$
\label{thm:stab1}
\end{thm} 

\begin{remarks}
The case where the unperturbed operator ${\bf H}$ has a kernel is
somewhat more involved but can be addressed similarly. There one must do another 
perturbation calculation near $\rho=0$ to understand how the zero eigenvalue(s) move with $\rho$ in order to compute $n_+({\bf H}_\rho)$ 
for $\rho$ small but non-zero. From there the calculation is the same: the number of positive eigenvalues can stay the same or decrease 
by one, and this is determined by the sign of $\langle {\bf 1},{\bf
  H}^{-1} {\bf 1} \rangle$.  Here ${\bf H}$ is singular but ${\bf 1}$
is in the range of ${\bf H}$, so ``${\bf H}^{-1} {\bf 1}$" may be interpreted
in the sense of the Moore-Penrose pseudo-inverse. 
\end{remarks}
As we remarked earlier there is actually a one-parameter family of
stationary solutions to the nonlocal equation Eqn. \eqref{eqn:allen-cahn}. 
We now give an alternative characterization of the stability criterion
by expressing it in terms of the dependence of the integrated reaction
rate on the mass of the solution.
\begin{corr}
Two alternative characterizations of the stability criterion in Theorem
(\ref{thm:stab1}) are as follows:
\begin{enumerate}
\item Let $M$ denote the total mass of the
solution
\[
M = \int_{-L}^L u \, dx 
\]
and $R$ denote the total reaction rate 
\[
R = \int_{-L}^L f(u) \, dx.
\]
Assume that the family of stationary solutions can locally be
parameterized by the total mass $M$, and that the number of positive
eigenvalues of the unperturbed operator is given by $n_+({\bf H})=k$. Then the dimension of the
unstable manifold is given by 
 \[
n_+({\bf {\widetilde H}}) = \left\{ \begin{array}{ll} k, & \qquad \frac{dR}{dM} <0 \\
k-1, & \qquad \frac{dR}{dM}>0 \end{array}\right.
\] 
\label{thm:stab2}
In particular a necessary condition for stability is that the total reaction rate must be an increasing function
of the total mass. 
\item Suppose that the stationary solution is given by the
quadrature
\[
\int_{\mu_-}^u \frac{dy}{\sqrt{2 E + 2 \kappa y - 2F(y)}} = x+L 
\]
for appropriate constants $E,\kappa$.
Let $\mu_-$ and $\mu_+$ be
two turning points for the quadrature: i.e. $\mu_{\pm}$ are simple roots of
$2 E + 2 \kappa u -2F(u)=0$ such that $2E + 2 \kappa u - 2F(u) >0$ for $u \in
(\mu_-,\mu_+)$, with $F$ the antiderivative of the reaction rate,
$F'(u)=f(u)$. 
Define the period type integrals 
\begin{align*}
& P(E,\kappa) = \frac{1}{2}\oint_{\Gamma} \frac{du}{\sqrt{2E + 2 \kappa u - 2F(u)}} \\
& M(E,\kappa) = \frac{1}{2}\oint_{\Gamma}  \frac{u du}{\sqrt{2E + 2 \kappa u - 2F(u)}} \\
& R(E,\kappa) = \frac{1}{2} \oint_{\Gamma}  \frac{f(u) du}{\sqrt{2E + 2 \kappa u - 2F(u)}},
\end{align*}
where $\Gamma$ is a simple closed contour containing the branch
cut along the real axis from $\mu_-$ to $\mu_+$.
Then the dimension of the unstable manifold is given by
\[
n_+({\bf {\widetilde H}}) = \left\{ \begin{array}{ll} k, & \qquad
                                                 \tau   <0 \\
k-1, & \qquad \tau >0 \end{array}\right.
\]
where $\tau$ is defined by 
\[
\tau = \frac{\frac{\partial
                                                       M}{\partial E}
                                                       \frac{\partial
                                                       P}{\partial \kappa} -
                                                      \frac {\partial
                                                       M}{\partial
                                                       \kappa}\frac{\partial
                                                       P}{\partial
                                                       E}}{\frac{\partial
                                                       R}{\partial E}
                                                       \frac{\partial
                                                       P}{\partial
                                                       \kappa} -
                                                       \frac{\partial
                                                       R}{\partial \kappa}
                                                       \frac{\partial
                                                       P}{\partial E}} 
\] and represents the rate of change of $M$ divided by the rate of
change of $R$ along the one-parameter family of stationary solutions. 
\end{enumerate}
\end{corr}
\begin{proof}
The proof consists of computing the family of stationary
solutions and observing that translation along the family of
stationary solutions generates the appropriate element of the range of
${\bf H}$ needed to compute $\langle {\bf 1},{\bf H}^{-1} {\bf 1} \rangle$. 

First, we show that we can find a one-parameter family of stationary
solutions through quadrature: see, for instance, the text of Landau
and Lifshitz\cite{LL}. A stationary solution $u$ must satisfy 
\[
0=u_{xx} + f(u) - \frac{1}{2L} \int_{-L}^L f(u) dx = u_{xx} + f(u) - \kappa
\] 
for some suitable $\kappa$; multiplying by $u_x$ we find that
\[ 0 = u_x (u_{xx} + f(u) - \kappa) = \frac{d}{dx} \left( \frac{1}{2} u_x^2 + F(u) - \kappa u \right) \Rightarrow  \frac{1}{2} u_x^2 + F(u) - \kappa u  = E,
\]
for some integration constant $E$, and $F'(u)=f(u)$. Solving for $u_x$ yields
\begin{equation}
\frac{du}{\sqrt{2 E + 2 \kappa u - 2F(u)}} = dx.
\end{equation}
The Neumann boundary condition $\frac{du}{dx} = 0$ is satisfied at a turning point of the function $\sqrt{2 E + 2\kappa u - 2F(u)}$; therefore we integrate along the real axis between two points $\mu_- = u(-L)$, $\mu_+ = u(L)$, at which $\sqrt{2 E + 2\kappa u - 2F(u)} = 0$:
\begin{align}
& \int_{\mu_-}^{\mu_+} \frac{du}{\sqrt{2 E + 2\kappa u - 2F(u)}} = \int_{-L}^L \, dx = 2L
\label{eqn:defP_real_axis}
\end{align}
or
\begin{align}
& P(E,\kappa) = \frac{1}{2}\oint_{\Gamma} \frac{du}{\sqrt{2 E + 2 \kappa u - 2F(u)}} = 2L,
\label{eqn:defP_contour_int}
\end{align}
where $\Gamma$ is a simple closed contour that loops around the segment between $\mu_-$ and $\mu_+$.  
The factor of $\frac12$ in the second equation arises in the
usual way due to the fact that the contributions from the top and the
bottom of the square root branch cut add. 
Condition (\ref{eqn:defP_contour_int}) defines a curve in the $(E,\kappa)$ plane along which we have a
stationary solution, defined locally by the vector field
\[
dE \frac{\partial P}{\partial E} + d\kappa \frac{\partial P}{\partial \kappa} =0.
\] 
If we choose to parameterize the curve $(E(s),\kappa(s))$ in the $(E,\kappa)$
plane by arc length $s$ then we can take 
\begin{align}
& \frac{dE}{ds} = -\frac{P_\kappa}{\sqrt{P_\kappa^2 + P_E^2}} \label{eqn:param1_dEds}
  \\
& \frac{d\kappa}{ds} = \frac{P_E}{\sqrt{P_\kappa^2 + P_E^2}}.
\label{eqn:param2_dkds}
\end{align}
Note that along this curve we have the identities 
\begin{align*}
&  R(E(s),\kappa(s)) = \frac12 \oint_{\Gamma} \frac{f(u) du}{\sqrt{2 E + 2 \kappa u - 2
  F(u)}} = \int_{-L}^L f(u(x;s)) dx \\
& M(E(s),\kappa(s)) = \frac12 \oint_{\Gamma} \frac{u du}{\sqrt{2 E + 2 \kappa u - 2
  F(u)}} = \int_{-L}^L u(x;s) dx \\
& P(E(s),\kappa(s))=2L. 
\end{align*}
Also note that one has the identity $\kappa P(E,\kappa) =
R(E,\kappa)$, so one could in principle eliminate one of these
quantities although we have chosen not to do so here.  

Having found a family of stationary solutions parameterized by the arc length, $u(x;s)$, we now proceed to compute the quantity $\langle {\bf 1},{\bf H}^{-1} {\bf 1}\rangle$ in terms of $M$ and $R$. 
First, take the equation for the stationary solution 
\[
u_{xx}(x;s) + f(u(x;s)) - \frac{1}{2L} \int_{-L}^L f(u(x;s))\, dx =0
\]
and differentiate with respect to the arc length parameter $s$, giving
\begin{equation}
u_{sxx} + f'(u) u_s -\frac{1}{2L} \int_{-L}^L f'(u) u_s \, dx =0
\label{eqn:stationary_diff_wrt_s}
\end{equation} 
which we recognize as (recalling that ${\bf H} v = v_{xx} + f'(u) v$)
\[
{\bf H} u_s = \frac{1}{2L} \frac{dR}{ds}. 
\] 
The right-hand side is a constant, thus we have 
\[
u_s = \left( \frac{1}{2L} \frac{dR}{ds}\right)   {\bf H}^{-1} {\bf 1}. 
\]
Integrating this identity gives 
\begin{equation}
 \frac{1}{2L} \frac{dR}{ds} \langle {\bf 1}, {\bf
  H}^{-1} {\bf 1} \rangle = \langle {\bf 1},u_s \rangle = \int_{-L}^L u_s \, dx = \frac{dM}{ds}
  \label{eqn:int_identity}
\end{equation}
or 
\begin{equation}
\langle {\bf 1},{\bf H}^{-1} {\bf 1}\rangle = 2L \frac{\frac{dM}{ds}}{\frac{dR}{ds}}.
\end{equation}
Since $L > 0$ is positive, the quantity $ \langle {\bf 1},{\bf H}^{-1}
{\bf 1}\rangle $ is positive if $R$ increases with increasing $M$ and is
negative if $R$ decreases with increasing $M$.  Applying the chain
rule and Equations (\ref{eqn:param1_dEds}) and (\ref{eqn:param2_dkds})  we find that 
\begin{equation}
\langle {\bf 1},{\bf H}^{-1} {\bf 1}\rangle 
= 2L \frac{\frac{\partial M}{\partial E} \frac{dE}{ds} +
                                                       \frac{\partial
                                                         M}{\partial
                                                         \kappa}
                                                       \frac{d\kappa}{ds}
                                                     }{\frac{\partial
                                                         R}{\partial E}
                                                       \frac{d E}{ds}
                                                       +
                                                       \frac{\partial
                                                         R}{\partial
                                                         \kappa}
                                                       \frac{d\kappa}{ds}}=
                                                     2L
                                                     \frac{\frac{\partial
                                                         M}{\partial
                                                         E}
                                                       \frac{\partial
                                                         P}{\partial \kappa} -
                                                       \frac{\partial
                                                         M}{\partial
                                                         \kappa}
                                                       \frac{\partial
                                                         P}{\partial
                                                         E}}{\frac{\partial
                                                         R}{\partial E}
                                                       \frac{\partial
                                                         P}{\partial
                                                         \kappa} -
                                                       \frac{\partial
                                                         R}{\partial
                                                         \kappa}
                                                       \frac{\partial
                                                         P}{\partial E}} = 2L \tau. 
\end{equation}
\end{proof}

While we are not aware of previous results of this type for
dissipative equations there are a number of results for the stability
of nonlinear dispersive waves that relate the index of some linearized
operator to the sign of the derivative of some conserved quantity with
respect to a parameter, analogous to the result presented here. The
classical Vakhitov-Kolokolov criteria, relating the stability of
solitary wave solutions $\psi(x,t) = e^{i \omega t} \phi(x;\omega )$ to the
nonlinear Schr\"odinger equation 
\[
i \psi_i = -\psi_{xx} + g(|\psi|^2) \psi 
\]
to the sign of $\frac{d}{d\omega} \int |\phi|^2(x;\omega) dx$ falls into
this category \cite{VK, GSSI, GSSII, PW}.
The problem of the stability
of periodic solutions to nonlinear dispersive waves is perhaps even
more similar to the present case: in this situation the index is
determined by the signs of certain determinants of derivatives of
period integrals (see \cite{KD, BJ, BJK} for examples).

It is worth discussing the relationship of this result to the results
of Freitas\cite{F95}, who
relates the dimension of the unstable manifold to
the lap number of the underlying solution. Specifically Freitas shows
(see Theorems 5.1 and 5.13 of \cite{F95}) that the  dimension of the unstable manifold
is related to $m$, the number of extrema  of the underlying solution
$u$ in $(-L,L)$ via 
\begin{align*}
m \leq n_+(\widetilde {\bf H}) \leq m+2 
\end{align*}
Thus the dimension of the unstable manifold can differ from the
number of extrema by up to two. 
To summarize the reasoning, the unstable manifold of the unperturbed operator has dimension equal to the lap number, $l(w) = m+1$;  the corresponding dimension for the perturbed operator can then go up (to $m+2$) or down (to $m$) by at most one. 

Here, we determine this direction (whether up to $m+2$, or down to $m$), by computing the index of the unperturbed operator; 
in terms of $m$, our result states that
\begin{align*}
m \leq n_+(\widetilde {\bf H}) \leq m+1,  \qquad & {\rm if} \; \langle {\bf 1},{\bf H}^{-1} 
  {\bf 1}\rangle >0 \\
m+1 \leq n_+(\widetilde {\bf H}) \leq m+2 \qquad & {\rm if} \; \langle {\bf 1},{\bf H}^{-1} {\bf 1}\rangle < 0, 
\end{align*}
where the sign of $\langle {\bf 1}, {\bf H}^{-1} {\bf 1}\rangle$ can be replaced
by the sign of $\frac{dR}{dM}$ or any of the other equivalent forms
discussed previously.

\section{Discussion}
In conclusion, we have presented a general method of analyzing low-rank perturbations of self-adjoint operators. 
We show how to use a
simple idea of classical differential geometry (the envelope of a family of 
curves) to completely analyze the spectrum. 
When the rank of the perturbation is two, this allows us to view the system in a geometric way through a phase diagram in the perturbation strengths $(\rho_1, \rho_2)$. By locating constant eigenvalue and eigenvalue coincidence curves (both computable through simple formulas), we can determine where the perturbed operator is stable, and where double real eigenvalues bifurcate into complex pairs.
This latter situation (bifurcation into a complex pair) coincides with a poorly conditioned eigenvalue, 
which in turn signals that small changes in the perturbation parameter will yield large changes in the operator behavior. 

We used these techniques to analyze three problems of this form; a 
model of the oculomotor integrator due to Anastasio and Gad\cite{ag07}, a continuum version of the oculomotor integrator model, and a nonlocal model of phase separation due to Rubinstein and Sternberg\cite{RS92}. In the first two problems, the physical interpretation of our model (a neural network that must maintain a steady eye position in the absence of input) required that we identify (a) where the perturbed system had a specific eigenvalue and (b) where this particular eigenvalue would be poorly conditioned. 
Our results in \S 1 show that both (a) and (b) can only occur in proximity to a specific point on the $(\rho_1, \rho_2)$ plane, which was then easy to visualize. In \S 2 and \S 3, some portions of the model are not completely specified by biology (such as the vestibular-to-Purkinje connections), but must be chosen arbitrarily (even randomly). In this paper, we analyzed a few carefully chosen examples. But, the geometric method we describe here also gives us rapid way to survey a large family of such models; using such a survey to draw conclusions about the vestibular-to-cerebellar pathway is an area for future work.

The problem analyzed in \S 4 involves a rank one (rather than rank two) perturbation, and so a phase plane approach is not applicable. Instead, we systematically exploit the rank one nature of the perturbation to characterize stability of stationary solutions in terms of the unperturbed operator.  We further show how to construct a one-parameter family of stationary solutions, and relate the stability condition to the relative change in two integrated quantities (mass and reaction rate) as one travels along this family.

In analyzing these three problems, we have by no means exhausted the possible applications. 
For example, the eigenvalue problem in 
Eqn. \eqref{eqn:reduce_match} is very similar to the stability 
problem for spike solutions to activator-inhibitor models in the 
limit of slow activator diffusion\cite{F94,IW02} (although 
the problem we study here differs because the eigenvalue enters in a non-linear way). Similar models 
of reaction-diffusion equations with non-local interactions have 
arisen in a number of other contexts including population dynamics\cite{FG89}, 
runaway ohmic heating\cite{C81,L95I,L95II}, and microwave heating \cite{BK98}. 
Therefore, we anticipate that
 the techniques presented here should be applicable to 
understanding these problems.

\section*{Competing Interests}
The authors have no competing interests.

\section*{Data Accessibility} 
Not applicable: the paper contains sufficient detail to reproduce the results.

\section*{Authors' Contributions}
TJA, AKB, and JCB designed the project. TJA originated the integrator model described in \S \ref{sec:numerics}; AKB and JCB originated the continuum integrator model (\S \ref{sec:continuum}) and analyzed the Rubinstein-Sternberg model described in \S \ref{sec:RS_model}. AKB and JCB created figures; TJA, AKB and JCB wrote the paper.

\section*{Acknowledgements}
We would like to thank two anonymous reviewers whose comments have helped us to improve the manuscript. 

\section*{Funding}
JCB and AKB received support from 
grant NSF-DMS0354462, and JCB from NSF-DMS0807584.



\clearpage

\begin{thebibliography}{10}

\bibitem{ag07}
Anastasio TJ, Gad YP.
\newblock Sparse cerebellar innervation can morph the dynamics of a model
  oculomotor neural integrator.
\newblock Journal of Computational Neuroscience. 2007;22(3):239--254.

\bibitem{RS92}
Rubinstein J, Sternberg P.
\newblock Nonlocal reaction-diffusion equations and nucleation.
\newblock IMA J Appl Math. 1992;48(3):249--264.

\bibitem{F94}
Freitas P.
\newblock A Nonlocal Sturm-Liouville Eigenvalue Problem.
\newblock Proceedings of the Royal Society of Edinburgh. 1994;124A:169--188.

\bibitem{IW02}
Iron D, Ward MJ.
\newblock The dynamics of multispike solutions to the one-dimensional
  {G}ierer-{M}einhardt model.
\newblock SIAM J Appl Math. 2002;62(6):1924--1951 (electronic).

\bibitem{C81}
Chafee N.
\newblock The electric ballast resistor: homogeneous and nonhomogeneous
  equilibria.
\newblock In: Nonlinear differential equations ({P}roc. {I}nternat. {C}onf.,
  {T}rento, 1980). New York: Academic Press; 1981. p. 97--127.

\bibitem{L95I}
Lacey AA.
\newblock Thermal runaway in a non-local problem modelling {O}hmic heating.
  {I}. {M}odel derivation and some special cases.
\newblock European J Appl Math. 1995;6(2):127--144.

\bibitem{L95II}
Lacey AA.
\newblock Thermal runaway in a non-local problem modelling {O}hmic heating.
  {II}. {G}eneral proof of blow-up and asymptotics of runaway.
\newblock European J Appl Math. 1995;6(3):201--224.

\bibitem{BK98}
Bose A, Kriegsmann GA.
\newblock Stability of localized structures in non-local reaction-diffusion
  equations.
\newblock Methods Appl Anal. 1998;5(4):351--366.

\bibitem{DH10}
Du Y, Hsu SB.
\newblock On a Nonlocal Reaction-Diffusion Problem Arising from the Modeling of
  Phytoplankton Growth.
\newblock SIAM J Math Anal. 2010;42(3):1305--1333.

\bibitem{F99}
Freitas P.
\newblock Nonlocal Reaction Diffusion Equations.
\newblock Fields Inst Comm. 1999;21:187--204.

\bibitem{BG}
Bruce JW, Giblin PJ.
\newblock Curves and Singularities.
\newblock Cambridge University Press; 1984.

\bibitem{spivak}
Spivak M.
\newblock Differential Geometry, Vol. III.
\newblock 3rd ed. Publish or Perish; 1999.

\bibitem{R89}
Robinson DA.
\newblock Control of eye movements.
\newblock In: Brooks VB, editor. Handbook of Physiology, Section 1: The Nervous
  System, Volume 2, Part 2. American Physiological Society; 1989. p.
  1275--1320.

\bibitem{r89b}
Robinson DA.
\newblock Integrating with neurons.
\newblock Annu Rev Neurosci. 1989;12:33--45.

\bibitem{bba08}
Barreiro AK, Bronski JC, Anastasio TJ.
\newblock Bifurcation theory explains waveform variability in a congenital eye
  movement disorder.
\newblock Journal of Computational Neuroscience. 2009;26:321--329.

\bibitem{BarreiroNN09}
Barreiro AK.
\newblock Mechanisms of neural integration: recent results and relevance to
  nystagmus modeling.
\newblock In: Harris CM, Gottlob I, Sanders J, editors. The Challenge of
  Nystagmus (Proceedings of the Second International Research Workshop on
  Nystagmus 2009). UK Nystagmus Network; 2012. .

\bibitem{BM85}
Berthoz A, Jones GM.
\newblock Adaptive Mechanisms in Gaze Control: Facts and Theories.
\newblock Elsevier; 1985.

\bibitem{GM76a}
Gonshor A, Jones JGM.
\newblock Short-term adaptive changes in the human vestibulo-ocular reflex arc.
\newblock Journal of Physiology. 1976;256:361--379.

\bibitem{GM76b}
Gonshor A, Jones JGM.
\newblock Extreme vestibulo-ocular adaptation induced by prolonged optical
  reversal of vision.
\newblock Journal of Physiology. 1976;256:381--414.

\bibitem{tsrz94}
Tiliket C, Shelhamer M, Roberts D, Zee DS.
\newblock Short term vestibulo-ocular reflex adaptation in humans. {I.}
  {E}ffect on the ocular motor velocity-to-position neural integrator.
\newblock Exp Brain Res. 1994;100:316--327.

\bibitem{Robinson74}
Robinson D.
\newblock The effect of cerebellectomy on the cat's vestibuloocular integrator.
\newblock Brain Res. 1974;71:195--207.

\bibitem{R76}
Robinson DA.
\newblock Adaptive gain control of vestibuloocular reflex by the cerebellum.
\newblock Journal of Neurophysiology. 1976;39:954--969.

\bibitem{b88}
B{\"{u}}ttner-Ennever JA.
\newblock Neuroanatomy of the Oculomotor System.
\newblock Amsterdam: Elsevier; 1988.

\bibitem{zybg81}
Zee S, Yamazaki A, Butler PH, G{\"{u}}cer G.
\newblock Effects of ablation of flocculus and paraflocculus on eye movements
  in primate.
\newblock J Neurophysiol. 1981;46:878--899.

\bibitem{cgrst90}
Chelazzi L, Ghirardi M, Rossi F, Strata P, Tempia F.
\newblock Spontaneous saccades and gaze holding ability in the pigmented rat.
  {II.} {E}ffects of localized cerebellar lesions.
\newblock Eur J Neurosci. 1990;2:1085--1094.

\bibitem{rambold02}
Rambold H, Churchland A, Selig Y, Jasmin L, Lisberger SG.
\newblock Partial ablations of the flocculus and ventral paraflocculus in
  monkeys cause linked deficits in smooth pursuit eye movements and adaptive
  modification of the VOR.
\newblock Journal of Neurophysiology. 2002;87:912--924.

\bibitem{NK03}
Nagao S, Kitazawa H.
\newblock Effects of reversible shutdown of the monkey flocculus on the
  retention of adaptation of the horizontal vestibulo-ocular reflex.
\newblock Neuroscience. 2003;118:954--969.

\bibitem{egv90}
Epema AH, Gerrits NM, Voogd J.
\newblock Secondary vestibulocerebellar projections to the flocculus and
  uvulo-nodular lobule of the rabbit: A study using HRP and double fluorescent
  tracer techniques.
\newblock Exp Brain Res. 1990;80:72--82.

\bibitem{lfsc85}
Langer T, Fuchs AF, Scudder CA, Chubb MC.
\newblock Afferents to the flocculus of the cerebellum in the rhesus macaque as
  revealed by retrograde transport of horseradish peroxidase.
\newblock J Comp Neurol. 1985;235:1--25.

\bibitem{lfcsl85}
Langer T, Fuchs AF, Chubb MC, Scudder CA, Lisberger SG.
\newblock Flocculuar efferents in the rhesus macaque as revealed by
  autoradiography and horseradish peroxidase.
\newblock J Comp Neurol. 1985;235:26--37.

\bibitem{Sekirnjak_etal_2003}
Sekirnjak C, Vissel B, Bollinger J, Faulstich M, {du Lac} S.
\newblock Purkinje cell synapses target physiologically unique brainstem
  neurons.
\newblock J Neurosci. 2003;23:6392--6398.

\bibitem{tg92}
Tan H, Gerrits NM.
\newblock Laterality in the vestibulo-cerebellar mossy fiber projection to
  flocculus and caudal vermis in the rabbit: A retrograde fluorescent
  double-labeling study.
\newblock Neurosci. 1992;47:909--919.

\bibitem{BV00}
Babalian AL, Vidal PP.
\newblock Floccular modulation of vestibuloocular pathways and
  cerebellum-related plasticity: an {\em in vitro} whole brain study.
\newblock Journal of Neurophysiology. 2000;84:2514--2528.

\bibitem{Keener}
Keener JP.
\newblock Principles of applied mathematics: transformation and approximation.
\newblock Addison Wesley, Advanced Book Program; 1988.

\bibitem{F94b}
Freitas P.
\newblock Bifurcation and stability of stationary solutions of nonlocal scalar
  reaction-diffusion equations.
\newblock J Dynam Differential Equations. 1994;6(4):613--629.

\bibitem{F95}
Freitas P.
\newblock Stability for Stationary Solutions for a Scalar Nonlocal Reaction
  Diffusion Equation.
\newblock Q J Mech Appl Math. 1995;48(4):557--582.

\bibitem{BF93}
Bates PW, Fife PC.
\newblock Spectral Comparison Principles for the {C}ahn-{H}illiard and phase
  field equations, and time scales for coarsening.
\newblock PHys D. 1993;53:990--1008.

\bibitem{MR0197830}
Magnus W, Winkler S.
\newblock Hill's equation.
\newblock Interscience Tracts in Pure and Applied Mathematics, No. 20.
  Interscience Publishers John Wiley \& Sons\, New York-London-Sydney; 1966.

\bibitem{MR3075381}
Eastham MSP.
\newblock The spectral theory of periodic differential equations.
\newblock Texts in Mathematics (Edinburgh). Scottish Academic Press, Edinburgh;
  Hafner Press, New York; 1973.

\bibitem{Simon2}
Simon B.
\newblock Orthogonal polynomials on the unit circle. {P}art 2 Spectral Theory.
  vol.~54 of American Mathematical Society Colloquium Publications.
\newblock American Mathematical Society, Providence, RI; 2005.
\newblock Spectral theory.

\bibitem{LL}
Landau LD, Lifshitz EM.
\newblock Course of theoretical physics. {V}ol. 1.
\newblock 3rd ed. Pergamon Press, Oxford-New York-Toronto, Ont.; 1976.
\newblock Mechanics, Translated from the Russian by J. B. Skyes and J. S. Bell.

\bibitem{VK}
Vakhitov NG, Kolokolov AA.
\newblock Stationary solutions of the wave equation in the medium with
  nonlinearity saturation.
\newblock Radiophys Quantum Electron;16:783--789.

\bibitem{GSSI}
Grillakis M, Shatah J, Strauss W.
\newblock Stability theory of solitary waves in the presence of symmetry. {I}.
\newblock J Funct Anal. 1987;74(1):160--197.

\bibitem{GSSII}
Grillakis M, Shatah J, Strauss W.
\newblock Stability theory of solitary waves in the presence of symmetry. {II}.
\newblock J Funct Anal. 1990;94(2):308--348.

\bibitem{PW}
Pego RL, Weinstein MI.
\newblock Eigenvalues, and instabilities of solitary waves.
\newblock Philos Trans Roy Soc London Ser A. 1992;340(1656):47--94.

\bibitem{KD}
Kapitula T, Deconinck B.
\newblock On the spectral and orbital stability of spatially periodic
  stationary solutions of generalized {K}orteweg--de {V}ries equations.
\newblock In: Hamiltonian partial differential equations and applications.
  vol.~75 of Fields Inst. Commun. Fields Inst. Res. Math. Sci., Toronto, ON;
  2015. p. 285--322.

\bibitem{BJ}
Bronski JC, Johnson MA.
\newblock The modulational instability for a generalized {K}orteweg-de {V}ries
  equation.
\newblock Arch Ration Mech Anal. 2010;197(2):357--400.

\bibitem{BJK}
Bronski JC, Johnson MA, Kapitula T.
\newblock An index theorem for the stability of periodic travelling waves of
  {K}orteweg-de {V}ries type.
\newblock Proc Roy Soc Edinburgh Sect A. 2011;141(6):1141--1173.

\bibitem{FG89}
Furter J, Grinfeld M.
\newblock Local vs. Nonlocal Interactions in Population Dynamics.
\newblock Journal of Mathematical Biology. 1989;27:65--80.

\end{thebibliography}

\section*{Appendix}
\subsection*{Derivation of Equations \eqref{eqn:AKformula_explicit_Mat} and \eqref{eqn:AKformula_explicit_SelfAdj}}
Here we include details for the derivation of Eqn. \eqref{eqn:AKformula_explicit_Mat}, for matrices.
We will make repeated use of the matrix determinant lemma:
\begin{equation} 
\det (\AM + \vec u \vec v \,^t)  =  (1 + \vec v \,^t \AM^{-1} \vec u) \det \AM
\label{eqn:MatDetLem}
\end{equation}
and the Sherman-Morrison formula:
\begin{equation} 
(\AM + \vec u \vec v\, ^t)^{-1}  = \AM^{-1} - \frac{\AM^{-1} \vec u \vec v\,^t \AM^{-1}}{1 + \vec v\,^t \AM^{-1} \vec u}
\end{equation}

We first apply Eqn. \eqref{eqn:MatDetLem} with $\AM = \MM-\lambda {\bf I} + \rho_1 \vec f_1 \vec g_1^t$, 
$\vec u = \rho_2 \vec f_2$, and $\vec v = \vec g_2$, resulting in Eqn. \eqref{eqn:applyMatDetLem1}. We then apply Sherman-Morrison to the inverse matrix, resulting in Eqn. \eqref{eqn:applySherMorr}. Finally we apply Eqn. \eqref{eqn:MatDetLem} a second time, but with $\AM = \MM-\lambda {\bf I}$, $\vec u = \rho_1 \vec f_1$, and $\vec v = \vec g_1$, and simplify to emphasize the polynomial form in $\rho_1, \rho_2$:
\begin{eqnarray}
\det(\widetilde{\bf M} - \lambda {\bf I}) & = & \det({\bf M} - \lambda {\bf I} + \rho_1 \vec f_1 {\vec g_1}^t +  \rho_2 \vec f_2 \vec g_2^t)
\nonumber \\
& = & \left( 1 + \rho_2 \vec g_2 ^t \left( {\bf M} - \lambda {\bf I} + \rho_1 \vec f_1 \vec g_1 ^t \right)^{-1} \vec f_2\right) \times \det({\bf M} - \lambda {\bf I} + \rho_1 \vec f_1 \vec g_1 ^t)
\label{eqn:applyMatDetLem1}\\
& = & \left( 1 + \rho_2 \vec g_2 ^t \left( ({\bf M} - \lambda {\bf I})^{-1}  -  \frac{({\bf M} - \lambda {\bf I})^{-1} \rho_1 \vec f_1 \vec g_1^t ({\bf M} - \lambda {\bf I})^{-1}}{1 + \vec g_1^t ({\bf M} - \lambda {\bf I})^{-1} \rho_1 \vec f_1} \right) \vec f_2\right) \times \det({\bf M} - \lambda {\bf I} + \rho_1 \vec f_1 \vec g_1 ^t) 
\nonumber \\
\label{eqn:applySherMorr}\\
& = & \left( 1 + \rho_2 \vec g_2 ^t \left( ({\bf M} - \lambda {\bf I})^{-1}  -  \frac{({\bf M} - \lambda {\bf I})^{-1} \rho_1 \vec f_1 \vec g_1^t ({\bf M} - \lambda {\bf I})^{-1}}{1 + \vec g_1^t ({\bf M} - \lambda {\bf I})^{-1} \rho_1 \vec f_1} \right) \vec f_2\right) 
\nonumber \\
& & \times \left( 1 + \vec g_1^t ({\bf M} - \lambda {\bf I})^{-1} \rho_1 \vec f_1 \right)  \times \det({\bf M} - \lambda {\bf I})
\label{eqn:applyMatDetLem2}\\
& = & \det({\bf M} - \lambda {\bf I}) \times \left[ 1 + \rho_1 \vec g_1^t ({\bf M} - \lambda {\bf I})^{-1} \vec f_1 + \rho_2 \vec g_2 ^t \left( {\bf M} - \lambda {\bf I} \right)^{-1} \vec f_2 \right.
\label{eqn:afterAlgebra} \\
&  &  \left. + \; \rho_1 \rho_2 \left( ( \vec g_1^t ({\bf M} - \lambda {\bf I})^{-1} \vec f_1 ) ( \vec g_2^t ({\bf M} - \lambda {\bf I})^{-1} \vec f_2 ) - ( \vec g_2 ^t ({\bf M} - \lambda {\bf I})^{-1}  \vec f_1 ) ( \vec g_1^t ({\bf M} - \lambda {\bf I})^{-1} \vec f_2 )  \right) \right]
\nonumber 
\end{eqnarray}
Using the cofactor formula $\cof^t {\bf A} = \det({\bf A}) {\bf A}^{-1}$ to replace each instance of $\det({\bf M}-\lambda {\bf I}) ({\bf M}-\lambda {\bf I})^{-1}$ with $\cof^t({\bf M}-\lambda {\bf I})$ gives the formula in Eqn \eqref{eqn:AKformula_explicit_Mat}.

%
To derive Eqn. \eqref{eqn:AKformula_explicit_SelfAdj}, use the spectral decomposition for self-adjoint operators:
\[ {\bf M} - \lambda {\bf I} = \sum_{i} (\lambda_i - \lambda) \vec \phi_i \vec \phi_i^t
\]
where $\lambda_i$ and $\vec \phi_i$ are the eigenvalue and eigenvectors of the unperturbed operator; substitute into Eqn. \eqref{eqn:afterAlgebra} and use the orthogonality of eigenfunctions ($\vec \phi_i^t \vec \phi_j = 0$ for $i \not=j$).

\subsection*{Derivation of Equation \eqref{eqn:const_evalue_wInv}}
We fill in a few more details for calculating eigenvalues of:
\[
\widetilde{\bf M}\vec w = {\bf M} \vec w + \rho_1 \vec f_1 \langle \vec g_1, \vec w\rangle + \rho_2 \vec f_2 \langle \vec g_2, \vec w\rangle = \lambda \vec w 
\]
where $\widetilde{\bf M}$ is an operator with compact resolvent.

First, act on this by the resolvent of the unperturbed operator, $R_{\lambda} = ({\bf M} - \lambda {\bf I})^{-1}$:
\begin{eqnarray}
0 & = & \left( {\bf M} -\lambda {\bf I} \right) \vec w + \rho_1 \vec f_1 \langle \vec g_1, \vec w\rangle + \rho_2 \vec f_2 \langle \vec g_2, \vec w\rangle  \Rightarrow \\
R_{\lambda} \, 0 & = & R_{\lambda} ({\bf M} -\lambda {\bf I}) \vec w + \rho_1 R_{\lambda} \vec f_1 \langle \vec g_1, \vec w\rangle + \rho_2 R_{\lambda} \vec f_2 \langle \vec g_2, \vec w\rangle \\
0 & = & \vec w + \rho_1 R_{\lambda} \vec f_1 \langle \vec g_1, \vec w\rangle + \rho_2 R_{\lambda} \vec f_2 \langle \vec g_2, \vec w\rangle \Rightarrow\\
-\vec w & = &  \rho_1 R_{\lambda} \vec f_1 \langle \vec g_1, \vec w\rangle + \rho_2 R_{\lambda} \vec f_2 \langle \vec g_2, \vec w\rangle
\end{eqnarray}
Now, act on the equation by both $\vec g_1$ and $\vec g_2$, to yield two consistency conditions:
\begin{eqnarray}
-\langle \vec g_1, \vec w\rangle & = & \rho_1 \langle \vec g_1, R_{\lambda} \vec f_1 \rangle \langle \vec g_1, \vec w\rangle + \rho_2  \langle \vec g_1, R_{\lambda} \vec f_2 \rangle \langle \vec g_2, \vec w\rangle \label{eqn:g1w_coupled}\\
-\langle \vec g_2, \vec w\rangle & = & \rho_1 \langle \vec g_2, R_{\lambda} \vec f_1 \rangle \langle \vec g_1, \vec w\rangle + \rho_2  \langle \vec g_2, R_{\lambda} \vec f_2 \rangle \langle \vec g_2, \vec w\rangle \Rightarrow \\
\langle \vec g_2, \vec w\rangle  & = & \frac{-\rho_1 \langle \vec g_2, R_{\lambda} \vec f_1 \rangle \langle \vec g_1, \vec w\rangle}{1 + \rho_2  \langle \vec g_2, R_{\lambda} \vec f_2 \rangle}  \label{eqn:g2w}
\end{eqnarray}
Substituting Eqn. \eqref{eqn:g2w} into Eqn. \eqref{eqn:g1w_coupled}  and dividing by the now common factor $\langle \vec g_1, \vec w\rangle$, will yield a single polynomial equation for $\rho_1$ and $\rho_2$:
\begin{eqnarray}
0 & = & 1 + \rho_1 \langle \vec g_1, R_{\lambda} \vec f_1 \rangle + \rho_2  \langle \vec g_2, R_{\lambda} \vec f_2 \rangle + \rho_1 \rho_2 \left(\langle \vec g_1, R_{\lambda} \vec f_1 \rangle  \langle \vec g_2, R_{\lambda} \vec f_2 \rangle - \langle \vec g_1, R_{\lambda} \vec f_2 \rangle  \langle \vec g_2, R_{\lambda} \vec f_1 \rangle\right) 
\end{eqnarray}
When $\vec g_1 = \vec g_2$, the final $\rho_1 \rho_2$ term is zero; this is what happens in the continuum example presented in \S 3. 

We now explain how to apply this formalism to the eigenvalue problem in \S 3:
\begin{eqnarray}
\beta \left( \Delta x\right)^2 \psi_{xx} + (-1 + 3 \beta) \psi + \frac{\rho_1\langle \psi,1\rangle}{\Delta x (\lambda+1)} \delta(x-x_1)+ \frac{\rho_2\langle \psi,1\rangle}{\Delta x (\lambda+1)} \delta(x-x_2) & = & \lambda \psi, ~~~~
\psi(0)=0=\psi(L) \nonumber \\
\label{reduce_appendix}
\end{eqnarray}
Here, ${\bf M} =  \beta \left( \Delta x \right)^2 \psi_{xx} +  (-1 + 3 \beta) \psi$ and (as we expect negative eigenvalues) we will act on Eqn. \eqref{reduce_appendix} with the resolvent operator $R_{\lambda} \equiv  \left( \beta \left( \Delta x \right)^2 \psi_{xx} +  (-1 + 3 \beta)\psi - \lambda \psi \right)^{-1}$. Here $\vec f_1 \propto \delta(x - x_1)$ and $\vec f_2 \propto \delta(x- x_2)$; therefore $R_{\lambda} \vec f_1$ is equivalent to solving the Green's function problem
\[ \beta \left( \Delta x \right)^2 \psi_{xx} + (-1 + 3 \beta) \psi - \lambda \psi  =  \delta(x-x_1); \quad \psi(0)=\psi(1)=0 \]
Since $\psi$ solves the  PDE  $\psi_{xx} + \frac{-1 - \lambda + 3 \beta}{ \beta \Delta x^2}\psi = 0$ for any $x \not= x_1, x_2$, $\psi$ must have the following form:
\begin{eqnarray}
\psi(x) & = & \left\{ \begin{array}{l l} A \sin \omega x, & x < x_1 \nonumber \\
B \sin \omega x + C \cos \omega x, & x_1 < x < x_2 \label{eqn:psi_Greens}\\
D \sin(\omega (L-x)), & x > x_2 \nonumber
\end{array} \right.
\end{eqnarray}
where $\omega^2 = \frac{-1 - \lambda + 3 \beta}{ \beta (\Delta x)^2}$.
By linearity of the resolvent, and the fact that $\vec g_1 = \vec g_2$, we can solve them together: i.e. $\rho_1 \langle \vec g_1, R_{\lambda} \vec f_1 \rangle + \rho_2 \langle \vec g_2, R_{\lambda} \vec f_2 \rangle = \rho_1 \langle \vec g_1, R_{\lambda} \vec f_1 \rangle + \rho_2 \langle \vec g_1, R_{\lambda} \vec f_2 \rangle =\langle \vec g_1,  \rho_1 R_{\lambda} \vec f_1 + \rho_2  R_{\lambda} \vec f_2 \rangle = \langle \vec g_1,  R_{\lambda} (\rho_1 \vec f_1 + \rho_2  \vec f_2) \rangle$.

Therefore we solve the Green's function problem only once; $\psi$ must satisfy
the following conditions, which impose continuity of $\psi$ at $x_1, x_2$ and the appropriate jump of the derivatives:
\begin{eqnarray}
\psi^+(x_1) & = & \psi^-(x_1)   \label{eqn:psi_cond1}\\
\psi^+(x_2) & = & \psi^-(x_2)   \label{eqn:psi_cond2}\\
\psi_x^+(x_1) - \psi_x^-(x_1) & = & -\frac{\rho_1}{\beta (\Delta x)^3 (\lambda + 1)} = -\frac{\rho_1}{\beta^2 (\Delta x)^3 (3 - (\Delta x)^2 \omega^2)}  \label{eqn:psi_cond3}\\
\psi_x^+(x_2) - \psi_x^-(x_2) & = & -\frac{\rho_2}{\beta (\Delta x)^3  (\lambda + 1)} = -\frac{\rho_2}{\beta^2 (\Delta x)^3 (3 - (\Delta x)^2 \omega^2)}  \label{eqn:psi_cond4}
\end{eqnarray}
This yields a system of 4 equations in the 4 unknowns $A$, $B$, $C$, and $D$. Similarly, we can express the action of $\vec g_1$ on any function of the form  Eqn. \eqref{eqn:psi_Greens} as a vector inner product: in this case, 
\begin{eqnarray}
\langle 1, \psi \rangle & = & \mathbf{g}^T \mathbf{w},\\
\mathbf{g} & = & \frac{1}{\omega}\left[ \begin{array}{c} 1 - \cos \omega x_1\\
 \cos \omega x_1 - \cos \omega x_2 \\ 
 \sin \omega x_2 - \sin \omega x_1 \\ 1 - \cos(\omega (L-x_2)) 
\end{array} \right]; \qquad \mathbf{w} = \left[ \begin{array}{c} A\\B\\C\\D \end{array} \right]
\end{eqnarray}
Thus our final polynomial is $0 = 1 + \mathbf{g}^T \mathbf{w}$, where $\mathbf{w}$ is the vector of coefficients we obtained by solving Eqn. (\ref{eqn:psi_cond1}--\ref{eqn:psi_cond4}).

\clearpage

\section*{Figure captions}

 \begin{figure}[h!]
\begin{center}
\includegraphics[height=3.5in]{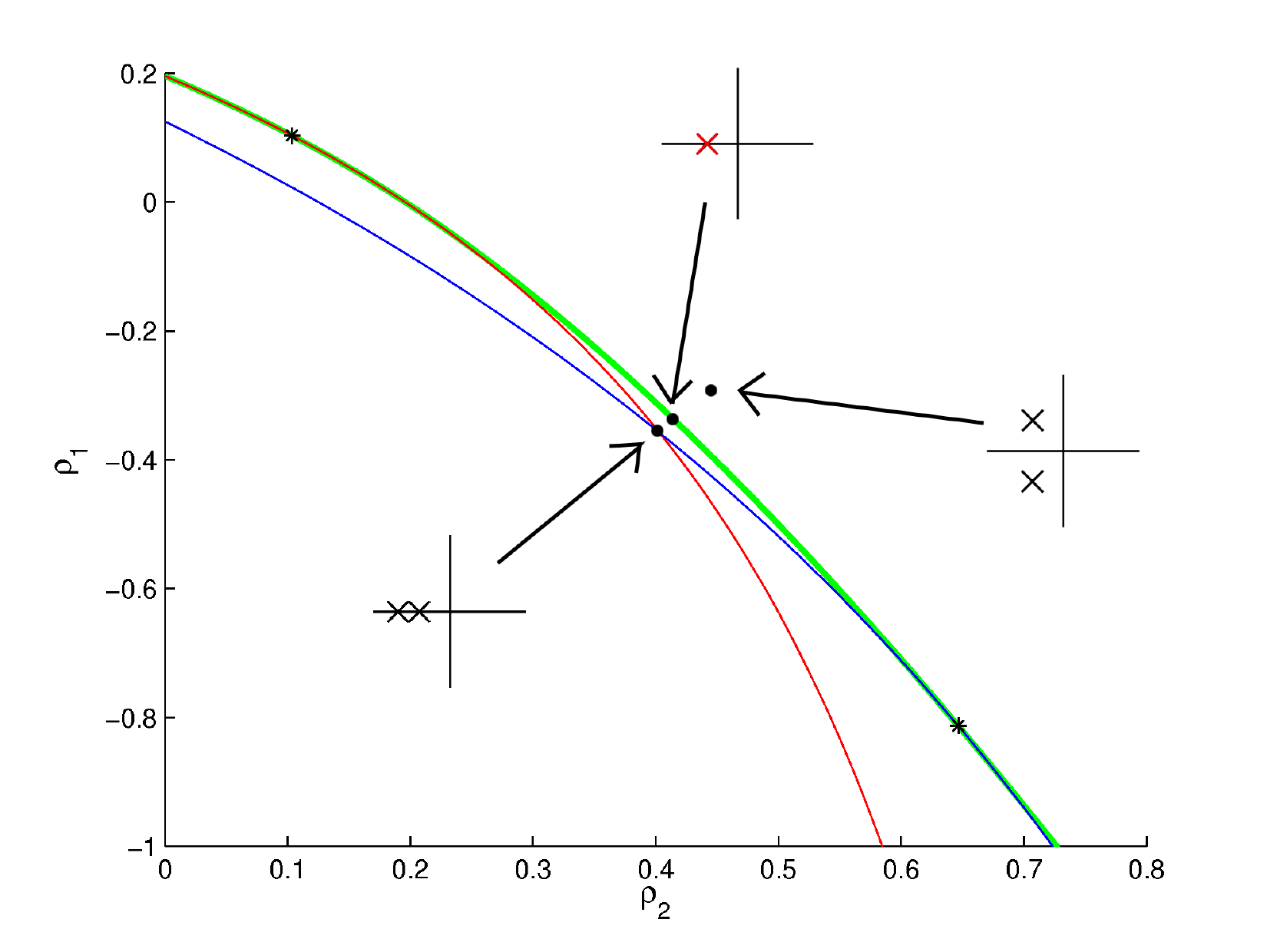}
\end{center}
\caption{ A schematic illustrating the bifurcation of eigenvalues across an envelope curve. The envelope curve (green bold solid) and two constant eigenvalue curves (blue and red light solid) are shown. The inset axes illustrate the relative positions of the eigenvalue pair in the vicinity of the bifurcation point.}
\label{fig:bifurcation}
\end{figure}

\begin{figure}[h!]
\begin{center}
\includegraphics[height=3in]{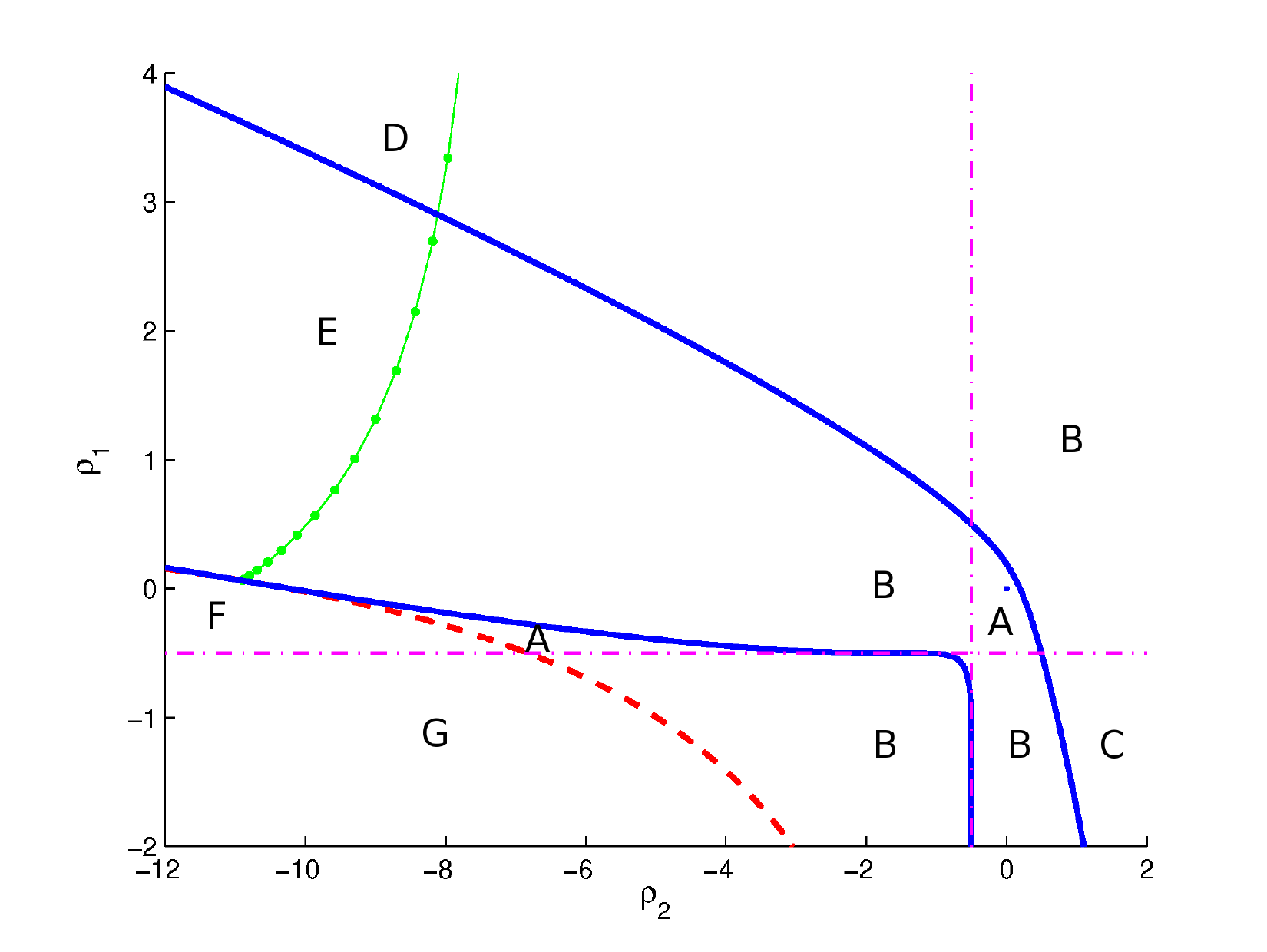}
\end{center}
\caption{The stability diagram in the $(\rho_2,\rho_1)$ plane 
for the model introduced in example 1. The bifurcation (blue), Hopf (green dotted), and zero eigenvalue (red dashed) curves are shown. The bifurcation curve also has a singular piece (magenta dot-dashed), for $\lambda = -2 + \frac{\sqrt{2}}{2}$, where the equations defining the envelope curve fail to have full rank.} 
\label{fig:StabDiag}
\end{figure}

\begin{figure}[h!]
\begin{center} 
\includegraphics{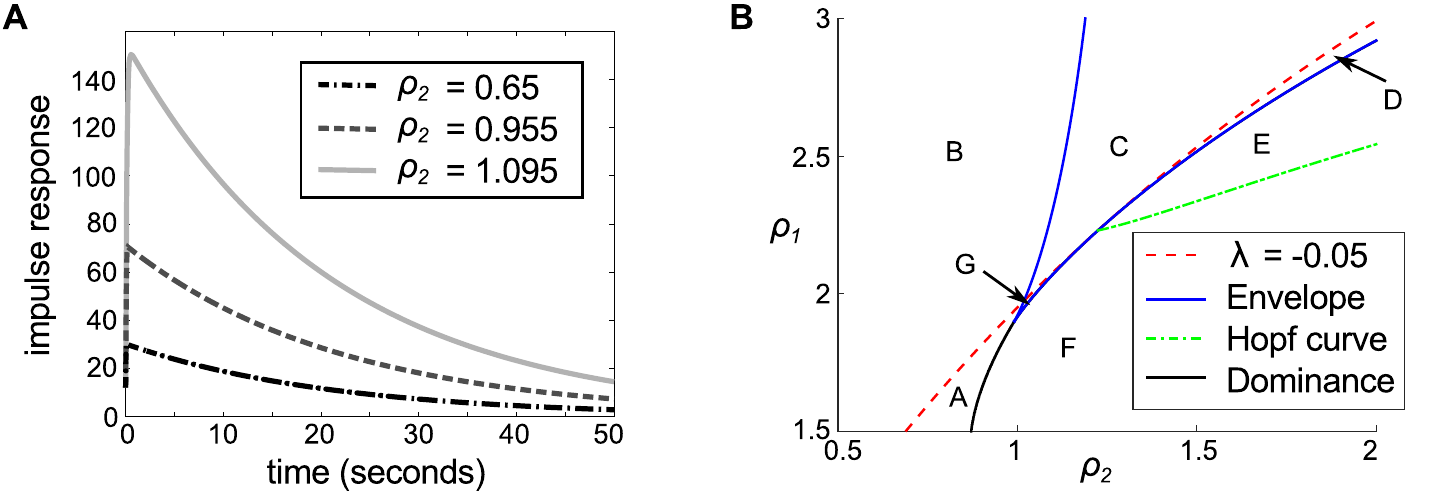}
\caption{(A) Impulse response of network as network attempts to increase gain
while maintaining $\lambda_{dominant}$. (B) Phase space of network showing normal operation. Letters indicate the relative positions of the \textit{three} most dominant eigenvalues. Region A: 1 real (dominant), 2 complex in the LHP. Region B: 1 real (dominant) in the RHP, 2 complex in the LHP. Region C: 1 real (dominant) in the RHP, 2 real in the LHP. Region D: 2 real (dominant) in the RHP, 1 real in the LHP. Region E: 2 complex (dominant) RHP, 1 real LHP. Region F: 2 complex (dominant) LHP, 1 real LHP. Region G: 3 real LHP.  We note that $\lambda_{dominant}$ is 
real unstable in B,C,D; complex unstable in E; complex stable in F; 
real stable in A,G.} \label{fig:norm_imp}
\end{center} 
\end{figure} 

\begin{figure}[h!]
\begin{center} 
\includegraphics{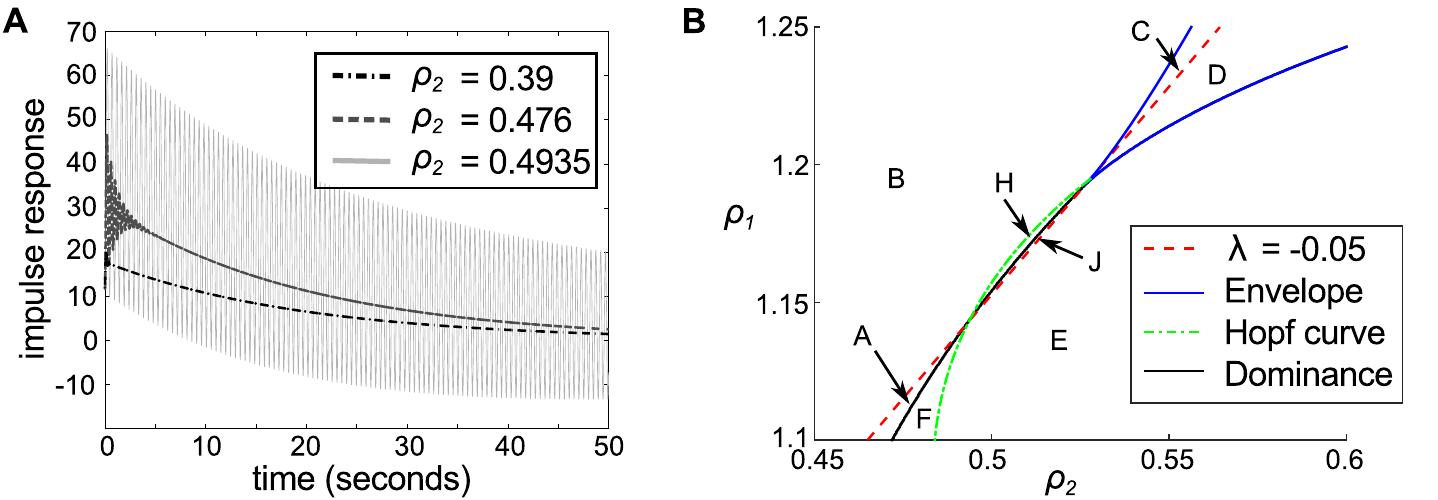}
\caption{(A) Impulse response of network as network attempts to increase gain
while maintaining $\lambda_{dominant}$. (B) Phase space of network showing pendular nystagmus. 
Letters indicate the relative positions of the \textit{three} most dominant eigenvalues; labels are as in Fig. \ref{fig:norm_imp}, with the addition of new regions ``H" and ``J".  Region A: 1 real (dominant), 2 complex in the LHP. Region B: 1 real (dominant) in the RHP, 2 complex in the LHP. Region C: 1 real (dominant) in the RHP, 2 real in the LHP. Region D: 2 real (dominant) in the RHP, 1 real in the LHP. Region E: 2 complex (dominant) RHP, 1 real LHP. Region F: 2 complex (dominant) LHP, 1 real LHP. Region H: 1 real (dominant) and 2 complex in the RHP. Region J: 2 complex (dominant) and 1 real in the RHP. (There is no Region G: 3 real LHP here).  We note that $\lambda_{dominant}$ is 
real unstable in B,C,D,H; complex unstable in E,J; complex stable in F; 
real stable in A.
} \label{fig:cn_imp}
\end{center} 
\end{figure} 

\begin{figure}[h!]
\begin{center}
\includegraphics{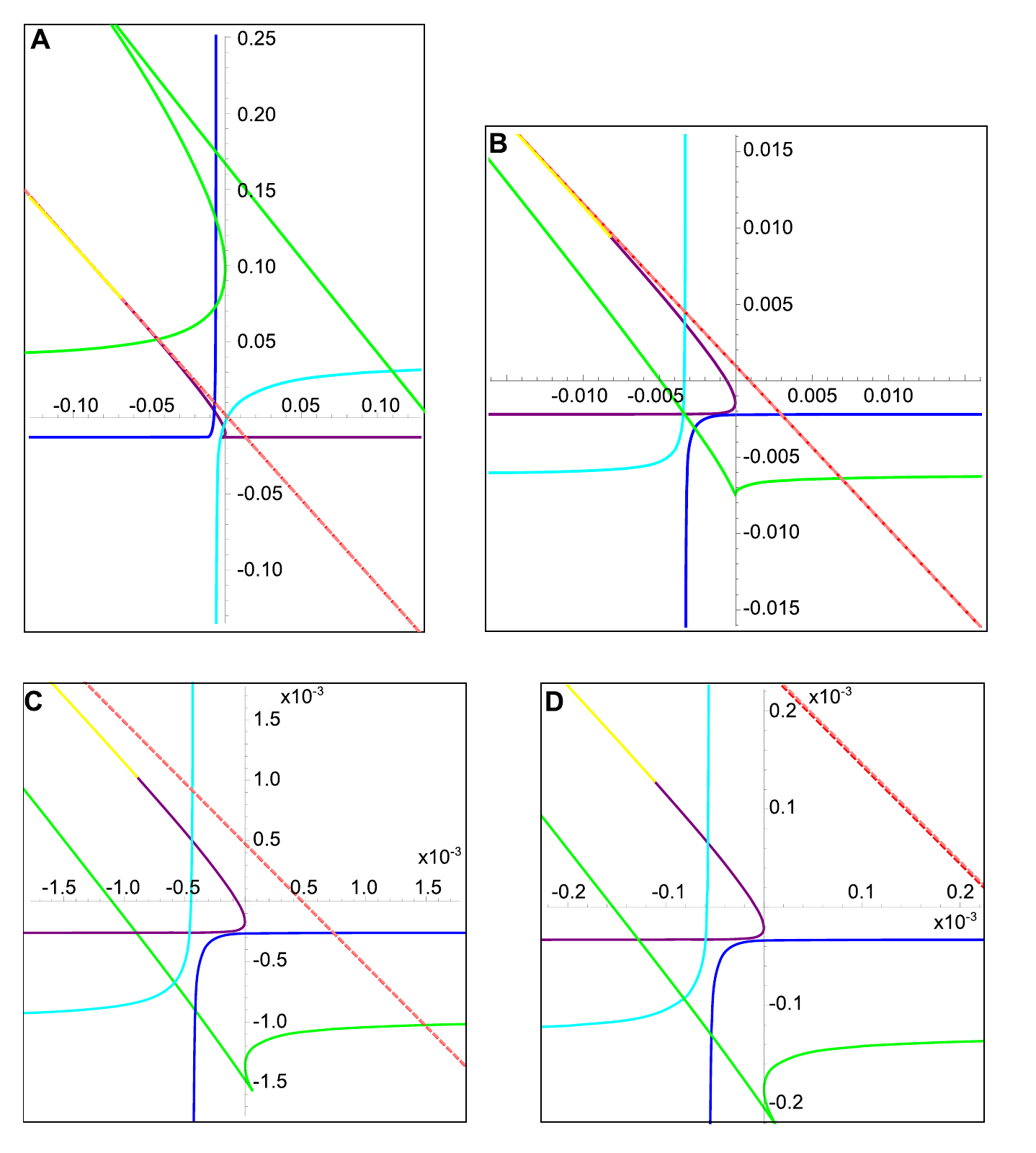}
\caption{Phase planes for the continuum oculomotor integrator model (Eqn. (\ref{eqn:psi_cont_def_match}--\ref{eqn:P1_cont_def_match})), for several values of $N$. Four pieces of the $\lambda < 0$ bifurcation curve are shown: $(0, 4\pi)$ (purple), $(4\pi, 6 \pi)$ (blue), $(6\pi, 8\pi)$ (cyan), $(8\pi, 12\pi)$ (green); $\lambda > 0$ curve (yellow).  Constant eigenvalue curves $\lambda = -0.05/200$ (red dashed) and $\lambda = 0$ (pink dashed) are visually indistinguishable. (A) $N= 12$. (B) $N = 24$. (C) $N=50$. (D) $N=100$.} \label{fig:continuous_envelope_finiteN} 
\end{center}
\end{figure}  

\end{document}